\documentclass[12pt]{amsart}
\usepackage[utf8]{inputenc}
\usepackage{lmodern}

\usepackage{pdflscape}
\usepackage{multirow}
\usepackage{graphicx}
\usepackage{bbm}

\usepackage{enumitem}
\usepackage{amssymb}
\usepackage{verbatim}
\usepackage{amsmath,amssymb}
\usepackage{amsmath}

\newcommand{\norm}[1]{ \big\lVert#1\big \rVert}

\usepackage{bigints}
\usepackage{enumitem}
\usepackage{hyperref}
\renewcommand{\Re}{\operatorname{Re}}
\renewcommand{\Im}{\operatorname{Im}}

\usepackage{amsfonts,array,amsmath,xcolor}
\usepackage{amsthm}

\usepackage{graphicx}
\numberwithin{equation}{section}

\newtheorem{definition}{Definition}[section]
\newtheorem{proposition}[definition]{Proposition}
\newtheorem{theorem}[definition]{Theorem}

\newtheorem*{proposition1}{Theorem 1.3$'$}

\newtheorem{lemma}[definition]{Lemma}
\newtheorem{corollary}[definition]{Corollary}
\newtheorem{remark}[definition]{Remark}
\title{Subcritical Fourier uncertainty principles
}
\author{}
\date{}

\author{Miquel Saucedo and Sergey Tikhonov}
\begin{document}

\address{ M. Saucedo \\
 Centre de Recerca Matem\`atica\\
Campus de Bellaterra, Edifici~C 08193 Bellaterra (Barcelona), Spain}
\email{miquelsaucedo98@gmail.com}

\address{ S. Tikhonov \\
 ICREA, Pg.
Llu\'is Companys 23, 08010 Barcelona, Spain
\\
Centre de Recerca Matem\`atica
Campus de Bellaterra, Edifici~C, 08193 Bellaterra (Barcelona), Spain,
and Universitat Aut\`onoma de
Barcelona,
Facultat de Ci\`encies,  08193 Bellaterra (Barcelona), Spain.}
\email{stikhonov@crm.cat}

\date{\today}



\thanks{M. Saucedo is supported by  the Spanish Ministry of Universities through the FPU contract FPU21/04230.
 S. Tikhonov is supported
by PID2020-114948GB-I00, 2021 SGR 00087. This work is supported by
the CERCA Programme of the Generalitat de Catalunya and the Severo Ochoa and Mar\'ia de Maeztu
Program for Centers and Units of Excellence in R\&D (CEX2020-001084-M)}
		
\begin{abstract}
It is well known that if a function $f$ satisfies
$$        \norm{f(x) e^{\pi \alpha |x|^2}}_p + \norm{\widehat{f}(\xi) e^{\pi \alpha |\xi|^2}}_q<\infty  \eqno{(*)}$$
with $\alpha=1$
and $1\le p,q<\infty$, then $f\equiv 0.$

We prove that if  $f$ satisfies $(*)$
with some $0<\alpha<1$ and $1\le p,q\leq \infty$,   then
$$ |f(y)|\le C
 (1+|y|)^{\frac{d}{p}}
 e^{- \pi \alpha |y|^2}, \quad
y\in \mathbb{R}^d,
$$
with $ C=C(\alpha,d,p,q)$
and this bound is sharp for $p\neq 1$.
We also study a related uncertainty principle for functions satisfying
$\;\;\displaystyle\norm{f(x)|x|^m}_p+  \norm{\widehat{f}(\xi)|\xi|^n}_q
<\infty.$

\end{abstract}
\maketitle

\section{Introduction}

Given an integrable function $f: \mathbb{R}^d\to \mathbb{C}$,
we define its Fourier transform by
$$ \widehat{f}(\xi)=\int_{\mathbb{R}^d} {f}(x) e^{-2\pi i \langle x, \xi \rangle} d x.
$$
We start with the Heisenberg uncertainty principle, discovered almost one hundred years ago, which is nowadays one of the fundamental results in mathematics and  quantum
mechanics, see \cite{folland}. Heisenberg's estimate reads as follows: if $f\in L_2(\mathbb{R}^d)$,  
 one has
\begin{equation}\label{heisenberg}
  \frac{d^2}{16\pi^2}\|f\|_2^4
\le\int_{\mathbb{R}^d}
|x |^2
|f(x)|^2 dx
\int_{\mathbb{R}^d}
|\xi |^2 |\widehat{f}(\xi)|^2 d\xi.
\end{equation}
This result
 quantitatively examines
the concepts of concentration around the origin and decay at infinity.
The equality in \eqref{heisenberg} is attained when $f$ is a
 Gaussian function, i.e,
 $f(x) = Ce^{-\lambda |x|^2}$, $\lambda>0$.

The fact that the Gaussian is the best localized function in time
and frequency was also recognized by  Hardy \cite{Hardy}.
In 1933, he proved  the following  result:
assume that  the estimates
\begin{equation}
\label{eq:hardyintro}
    f(x)=O\big(|x|^m e^{-a\pi |x|^2}\big)\quad \mbox {and}\quad  \widehat{f}(\xi)
=O\big(|\xi|^n e^{-b\pi |\xi|^2}\big)
\end{equation}
 hold for $|x|,|\xi| \to \infty $
  with some positive $a,b$ and non-negative $m,n$, then $f=0$ whenever
$ab > 1$; and $f(x) = P(x)e^{-a\pi |x|^2}$  with  a polynomial $P$,
if $ab = 1$. (Hardy only studied the one-dimensional case, for the multivariate result see \cite{bonami, than}.)
Both Heisenberg's and Hardy's
uncertainty principles
 show that
 a function and its  Fourier transform
cannot both be highly localized.

This result was later generalized to the $L_p$ setting by Cowling and Price \cite{cowling}, who obtained the following result:
if $1\le p,q\le \infty$, $(p,q)\ne (\infty,\infty)$ and a function $f$ satisfies 
\begin{equation}\label{cowling}
\norm{f(x)e^{\pi |x|^2}}_p + \norm{\widehat{f}(\xi)e^{\pi |\xi|^2}}_q < \infty, 
\end{equation}
then, necessarily, $f$ is identically zero.


Let us now consider   the subcritical case, that is, $ab<1$. It is clear that in this case
 the class of functions $f$
 satisfying \eqref{eq:hardyintro}
 is an infinite dimensional space, which, in particular,  contains  the Hermite functions. On the other hand,
Vemuri \cite{vemuri} obtained that  
 condition \eqref{eq:hardyintro} with $ab<1$ implies a rather strong property on $f$, namely,
  the exponential decay of its Hermite coefficients. See also  \cite{kulikov2023gaussian, RR}.

Very recently, Kulikov,  Oliveira, and Ramos \cite{kulikov2023gaussian}
 have studied the subcritical case of \eqref{cowling} and proved the following  result: assume that $0<\alpha<1$, $1\le p,q< \infty$, and
$f$ satisfies
\begin{equation}
\label{eq:mainequation--} \norm{f(x) e^{\pi \alpha |x|^2}}_p + \norm{\widehat{f}(\xi) e^{\pi \alpha |\xi|^2}}_q  < \infty,
\end{equation}
then,
  for
each $\varepsilon > 0$, there is a positive constant $A = A(\varepsilon)$ such that
$$| f (x)| \le A(\varepsilon) e^{-(1-\varepsilon)\alpha \pi|x|^2}.
$$
Thus, it is a natural question to find the
sharp
rate of decay at infinity
for functions satisfying (\ref{eq:mainequation--}).
More specifically, the following question was asked in  \cite{kulikov2023gaussian}:
 is it true that (\ref{eq:mainequation--}) for $p=q=2$ implies
$$| f (x)| \lesssim  e^{- \pi\alpha|x|^2}\quad?$$

 In this paper, we answer these questions
 through the following uncertainty principle:

\begin{theorem}
\label{theorem:mainth}
Let $0<\alpha<1$ and $1 \le p,q\le\infty$, $p\ne 1$. Then$,$ for $y\in\mathbb{R}^d,$
\begin{equation}
\label{eq:mainequation}\sup_{f\in L_1+L_2}\frac{|f(y)|}{\norm{f(x) e^{\pi \alpha |x|^2}}_p + \norm{\widehat{f}(\xi) e^{\pi \alpha |\xi|^2}}_q} \approx _{\alpha, d, p,q}  (1+|y|)^{\frac{d}{p}}e^{-\pi \alpha |y|^2}.\end{equation}
\end{theorem}

A routine application of the uniform boundedness principle shows that Theorem \ref{theorem:mainth} is equivalent to the following statement, which answers negatively the question in \cite{kulikov2023gaussian}:

\begin{corollary}
\label{theorem:equivalency}
   Let $0<\alpha<1$ and $1 \le p,q\le\infty$, $p\ne 1$. Then,
    \begin{equation}
      \label{eq:decay}
      \sup _{y\in \mathbb{R}^d} |f(y)| e^{ \pi \alpha |y|^2} E(|y|) <\infty
  \end{equation} for any function $f:\mathbb{R}^d \to \mathbb{C}$ for which
  \begin{equation}
      \label{eq:space}
       \norm{f(x) e^{\pi \alpha |x|^2}}_p + \norm{\widehat{f}(\xi) e^{\pi \alpha |\xi|^2}}_q<\infty
  \end{equation}
       if and only if
       $$\sup _{y\in \mathbb{R}^d} E(|y|) (1+|y|)^{\frac{d}{p}}<\infty.$$
\end{corollary}

The technique we use to prove Theorem
\ref{theorem:mainth} allows us to obtain
 a curious inequality, which is closely related to Theorem
\ref{theorem:mainth} (see item {\it(iii)} in Remark \ref{remark:intro}).
This result can be seen as a
Heisenberg type uncertainty principle for different norms
or, for $p=q=2$, as
a weighted Landau--Kolmogorov type estimate.

\begin{theorem}
\label{theorem:main2}
Let $1\leq p,q \leq \infty$, $m\in \mathbb{R}_+$, and $\varepsilon>0$ be such that $m-d/p>\varepsilon$.
Set $n-d/q'=m-d/p$. Then, there exists a constant $C:=C(\varepsilon,d,p,q)$, independent of $m$, such that for any $y\in \mathbb{R}^d$
we have
    \begin{equation}
    \label{eq:y2}
        |f(y)|(1+|y|^{m}) \leq C (1+|y|^{d/p})\left(\norm{f(x)|x|^m}_p+  \norm{\widehat{f}(\xi)|\xi|^n}_q\right)
    \end{equation}
and
    \begin{equation}
    \label{eq:m}
        |f(y)|(m!)^{1/2} \leq C^m  \left(\norm{f(x)|x|^m}_p+  \norm{\widehat{f}(\xi)|\xi|^n}_q\right).
    \end{equation}

\end{theorem}

Several comments are in order.
\begin{remark}
 \label{remark:intro}\begin{enumerate}[label=(\roman*),itemjoin=\\,mode=unboxed]
    \item
   {\textnormal{
   By a scaling argument, relation \eqref{eq:mainequation} in Theorem \ref{theorem:mainth}
  can be replaced by
      $$
\sup_{f\in L_1+L_2}\frac{|f(y)|}{\norm{f(x) e^{\pi \alpha |x|^2}}_p + \norm{\widehat{f}(\xi) e^{\pi \beta |\xi|^2}}_q} \approx _{\alpha, \beta, d, p,q}  (1+|y|)^{\frac{d}{p}}e^{-\pi \alpha |y|^2}.$$
     for $0<\alpha\beta <1$.
    }}

     \item {\textnormal{
   In the proof of the
     "$\lesssim$" part of Theorem \ref{theorem:mainth}, the restriction
     $p\ne 1$ is not
     necessary.
    }}

     \item {\textnormal{
   In the proof of the
     "$\lesssim$" part of Theorem \ref{theorem:mainth}, the restriction
     $\alpha<1$ is not necessary.
This gives us a direct proof of the Cowling-Price result from Hardy's uncertainty principle.
 Indeed, if $f$ satisfies
         \begin{equation}
    \label{up--infty}
     \norm{f(x)e^{\pi |x|^2}}_p + \norm{\widehat{f}(\xi)e^{\pi |\xi|^2}}_q < \infty,
         \end{equation}
     then one has simultaneously  $|f(x)|\lesssim (1+|x|)^{\frac{d}{p}} e^{-\pi |x|^2}$ and $|\widehat{f}(\xi)|\lesssim (1+|\xi|)^{\frac{d}{q}} e^{-\pi |\xi|^2}$. Hence, using Hardy's result we conclude that $f(x)=e^{-\pi |x|^2}P(x),$ where $P$ is  a polynomial.
     Finally, since
       \eqref{up--infty}
       holds for $(p,q)\ne (\infty,\infty)$,
     we deduce that $P\equiv 0$. We also note that further connections  between Hardy's and Cowling-Price's results are discussed in Section 3 of \cite{vega}, where
  real variable proofs
of both  uncertainty principles
were found. See also \cite{mal}.
  }}
    \item
    {\textnormal{ In the case $p=q=2$, inequality \eqref{eq:y2} in Theorem \ref{theorem:main2} implies the  "$\lesssim$" direction of Theorem \ref{theorem:mainth}. Indeed,  for $|y|\ge 1$ we have
\begin{eqnarray}
\nonumber
\qquad\quad |f(y)|^2 e^{2 \pi \alpha |y|^2} & \lesssim_{\alpha,d}& |f(y)|^2 \sum_{m=\lceil{d/2\rceil}+1}^\infty \frac{(2 \pi \alpha |y|^2)^m}{m!}  \\
\nonumber
&\lesssim_{\alpha,d}&
(1+ |y|^{d}) \sum_{m=\lceil{d/2\rceil}+1}^\infty \frac{(2 \pi \alpha)^m}{m!} \left(\norm{f(x) |x|^m}_2^2 + \norm{\widehat{f}(\xi) |\xi|^m}_2^2
\right) \\
\label{eq:expo}
 & \lesssim_{\alpha,d}& (1+ |y|^{d})\left(\norm{f(x) e^{ \pi \alpha |x|^2}}_2^2 + \norm{\widehat{f}(\xi) e^{ \pi \alpha |\xi|^2}}_2^2\right).
\end{eqnarray}
  }}

\item
    {\textnormal{
   Analyzing the proof of Theorem \ref{theorem:mainth}, we see that for any fixed $\lambda>0$,  $\norm{f(x) e^{\pi \alpha |x|^2}}_p$ in 
    \eqref{eq:mainequation}
    can be replaced by $\displaystyle\left(\int_{y+  \lambda y [ -1,1]^d} |f(x)|^p e^{p\pi \alpha |x|^2} dx\right)^{\frac{1}{p}}. $
Similarly, in Theorem \ref{theorem:main2},  $\norm{f(x) |x|^m}_p$ can be replaced by 
$$\qquad\,\,\left(\int_{y+  \frac{\lambda m}{y} [ -1,1]^d} |f(x)|^p |x|^{mp} dx\right)^{\frac{1}{p}} \quad \mbox{and} \quad\left(\int_{y+  \sqrt{\lceil m + n \rceil} [ -1,1]^d} |f(x)|^p |x|^{mp} dx\right)^{\frac{1}{p}} $$ in inequalities \eqref{eq:y2} and \eqref{eq:m}, respectively. }}
\item
    {\textnormal{
Using the previous item, we obtain a quantitative version of the Cowling--Price uncertainty principle, where the result depends on the behaviour of 
$\displaystyle\int_{- y}^{y} |f(x)|^p e^{p\pi  |x|^2} dx$ under the assumption $\norm{\widehat{f} e^{\pi \xi^2}}_q < \infty$.
See Proposition \ref{quantitativeCP} for details. 
This continues the research 
initiated  
in \cite{nachr}.
  }}
\item
    {\textnormal{In the case when 
both $f(x)e^{\pi \alpha x^2}$ and $\hat{f}(\xi)e^{\pi \alpha \xi^2}$ are tempered distributions, an
 uncertainty principle related to Theorem \ref{theorem:mainth}  was studied in \cite{demange}.
  }}
  \end{enumerate}
  \end{remark}


{\it Notation.}
\\
Throughout this paper,
we will denote by $C > 0$ absolute constants that may change from
line to line.
We will often use the symbol $F \lesssim G$  to mean that $F \le C G$.
The symbol $F \approx G$  means that both $F \lesssim G$ and $G \lesssim F$. If a constant is allowed to depend on a given parameter, the parameter dependence will be described in parenthesis, e.g., $C(\lambda_1,\dots, \lambda_l)$. Then
$F\lesssim _{\lambda_1, \dots, \lambda_l} G$ means that $F \leq C(\lambda_1,\dots, \lambda_l) G$.

By $\mathbbm{1}_E$ we  denote the characteristic function of the set $E$.
By  $\lfloor x \rfloor$ and $\lceil x \rceil$ we denote the floor and ceiling functions  of $x$, respectively.
Also,  $\frac1p+\frac1{p'}=1$ for $1\le p\le\infty$.

\vspace{1mm}

{\it Outline of the paper.}
 \\
 In Section \ref{Section 2}, we prove
 a weak form  of inequality \eqref{eq:y2} of Theorem \ref{theorem:main2}, where we  allow the constant $C$ to depend on $m$.
This result  can be treated  as a  model case and is based on the following duality relation: 
$$\sup_{f\in L_1+L_2} \frac{|f(y)|}{\norm{fW}_2+ \norm{\widehat{f}V}_2} \leq\inf_ {\phi \in L^1}\big(\norm{(1- \widehat{\phi}(\xi))V^{-1}(\xi)}_2 + \norm{\phi(y-x ) W^{-1}(x)}_2\big),
$$ which holds under some minor conditions on the  weights $W$ and $V$.
Building on this inequality,  we proceed by choosing a function $\phi$ for which the right-hand side is small.
In Section \ref{Section 3}
we obtain the properties of $\phi$ which we need to prove
  the upper bound in Theorem \ref{theorem:mainth} and
 Theorem \ref{theorem:main2}.
The proof of the lower bound  in Theorem \ref{theorem:mainth}
is contained in Section \ref{Section 4}.
 Our approach is motivated by
the special case
 $\alpha={1}/{\sqrt{2}}$ and $p=q=2$.
 In this case, inequality  \eqref{eq:mainequation} is proved to be equivalent to a point evaluation estimate for bandlimited functions.
Finally, in Section \ref{Section 5},  we collect further applications of our results   
and  open questions.

\vspace{1mm}

{\it Acknowledgements.}
 \\
We would like to thank Aline Bonami for drawing our attention to the results of Demange \cite{demange}.
 
\vspace{3mm}
\section{Motivation
} \label{Section 2}

\subsection{Duality lemma and weak version of Theorem \ref{theorem:main2} for $p=q=2$}
In this section, we obtain Theorem \ref{theorem:main2} in a weaker form, where we allow the constant $C$  in
 \eqref{eq:y2} to depend on $m$.
 This provides us with several hints on how  to attack the "$\lesssim$" part in Theorem  \ref{theorem:mainth} and
 Theorem \ref{theorem:main2}.

Observe that the quantity $\norm{f(x) |x|^m}_2+ \norm{\widehat{f}(\xi) |\xi|^m}_2$ can be regarded as a norm on the intersection of the spaces $L_{2,m}$ and $\widehat{L}_{2,m}$ with the norms $\norm{f(x) (1+|x|^m)}_2$ and $\norm{\widehat{f}(\xi) (1+|\xi|^m)}_2$, respectively. Then, inequalities \eqref{eq:y2} and \eqref{eq:m} become upper bounds for the norm of the functional $\delta_y(f)= f(y)$ as an element of $\left(L_{2,m} \cap \widehat{L}_{2,m}\right)^* = L_{2,m}^* + \widehat{L}_{2,m}^*$ (see, e.g., \cite[p. 174]{bs}), that is, for the quantity $$\inf_{\phi_1+ \phi_2 = \delta_y} \norm{\phi_1(x) (1+|x|^{m})^{-1}}_{2} +\norm{\widehat{\phi}_2(\xi) (1+|\xi|^{m})^{-1}}_{2}. $$

This idea leads us to the following statement:




\begin{lemma}
\label{lemma:dual}
    Suppose  $V$ and $W$ are  radial  non-negative weight functions
    on $\mathbb{R}^d$ and
    as radial functions non-decreasing on $(0,\infty)$.
Assume further that
      \begin{equation}
    \label{condition1}
   \norm{\mathbbm{1}_{|x|\geq 1}W^{-1}(x)}_2+ \norm{\mathbbm{1}_{|\xi|\geq 1} V^{-1}(\xi)}_2<\infty.
        \end{equation}
        Then, for any $y\in \mathbb{R}^d$,
    \begin{eqnarray}
   \nonumber
        \sup_{f\in L_1+L_2} \frac{|f(y)|}{\norm{fW}_2+ \norm{\widehat{f}V}_2} &\leq& \inf_ {\phi \in L^1}
        \big(
        \norm{\left(e^{2 \pi i \langle\xi, y\rangle}- \widehat{\phi}(\xi)\right)V^{-1}(\xi)}_2 + \norm{\phi(-x) W^{-1}(x)}_2\big)\\
         \label{eq:duals}
        &=&\inf_ {\phi \in L^1}\big(\norm{(1- \widehat{\phi}(\xi))V^{-1}(\xi)}_2 + \norm{\phi(y-x ) W^{-1}(x)}_2\big).
    \end{eqnarray}
\end{lemma}
\begin{proof}
To begin with, note that, by the Amrein--Berthier theorem (see \cite{am, jaming}),  one has
     \begin{equation}
    \label{Amrein}
\int_{|x|\leq 2}\big( |f(x)|^2 + |\widehat{f}(x)|^2
\big) dx \lesssim
\int_{|x|\geq 2}\big( |f(x)|^2 + |\widehat{f}(x)|^2
\big) dx.
 \end{equation}
   This fact and condition \eqref{condition1} imply that any function $f$ satisfying 
   $\norm{fW}_2+ \norm{\widehat{f}V}_2<\infty$ is such that  $f,\widehat{f}\in L^1(\mathbb{R}^d)$. Therefore, using the Fourier inversion formula and the
 Cauchy--Schwarz inequality, we have
   \begin{eqnarray*}
        |f(y)|
&\le&\Big|\int_{\mathbb{R}^d} \left(e^{2 \pi i \langle \xi, y \rangle}- \widehat{\phi}(\xi)\right) \widehat{f}(\xi) d \xi\Big|+ \Big|\int_{\mathbb{R}^d} \phi(-x)f(x)dx \Big|\\
&\leq& \Big(\norm{\big(e^{2 \pi i \langle\xi, y\rangle}- \widehat{\phi}(\xi)\big)V^{-1}(\xi)}_2 + \norm{\phi(-x) W^{-1}(x)}_2\Big) \Big(\norm{fW}_2+ \norm{\widehat{f}V}_2\Big),
   \end{eqnarray*}
   completing the proof.

\end{proof}

  For some choices of $V$ and $W$ it is fairly simple to derive a sharp estimate for the infimum on the right-hand side of \eqref{eq:duals}.
 In particular, we illustrate a use of Lemma \ref{lemma:dual} by the following weaker version of Theorem \ref{theorem:main2}:

  \begin{proposition1}
  \label{prop:toy}
      Let $m >\frac{d}{2}$. Then,
for any $y\in\mathbb{R}^d$,      \begin{equation}
\label{eq:toy}\sup_{f\in L_1+L_2}\frac{|f(y)|}{\norm{f(x) |x|^m}_2+ \norm{\widehat{f} (\xi)|\xi|^m}_2} \approx _{m,d}  (1+|y|)^{-m+\frac{d}{2}}.
\end{equation}
  \end{proposition1}
  \begin{proof}
In light of \eqref{Amrein}, we have
  $$\norm{f(x) |x|^m}_2+ \norm{\widehat{f}(\xi) |\xi|^m}_2 \approx_d \norm{f W_m}_2+ \norm{\widehat{f} W_m}_2,$$ where $W_m(x)=1+|x|^m$.
Observing that \eqref{condition1} holds by H\"{o}lder's inequality,
we now apply  Lemma \ref{lemma:dual} with $V=W=W_m$.

  For simplicity, we prove the result for $d=1$ and $y>0$.

  For the "$\lesssim$" part in \eqref{eq:toy}, let $\widehat{\phi}_1$ be a non-negative smooth function supported in $[-2,2]$ such that $\widehat{\phi_1}(\xi)=1$ for $|\xi|\leq 1$. For any $y>0$, set
  $\phi_y(x):=y \phi_1(yx)$.
  Then $\widehat{\phi_y}(\xi)=\widehat{\phi_1}(\xi/y)$.  \newline
  In order to apply Lemma \ref{lemma:dual}, we estimate
  $$\norm{(1- \widehat{\phi}_y)W_m^{-1}}_2 \lesssim \left(\int_{y}^\infty x^{-2m} dy\right)^{\frac{1}{2}} \approx _my^{-m + \frac{1}{2}}.$$

  Further, since $\phi_1$ is a Schwartz function, we have $\displaystyle|\phi_1(x)|\lesssim_m  \frac{1}{1+ |x|^{m+2}}$. Therefore, a simple computation shows that, for $y$ large enough, we have  \begin{eqnarray*}
   y^{-2}\norm{\phi_y(y-\cdot )
  W_m^{-1}}^2_2 &\lesssim_m&  \int_{-\infty}^\infty (1+ |x|^m)^{-2} (1+ |y (y-x)|^{m+2})^{-2} dy
 \lesssim_m y^{-2m-1},\end{eqnarray*}
whence the result follows for $y$ large enough. For small $y$ we use the fact that condition \eqref{condition1} implies that $f,\widehat{f} \in L ^1$ together with the inequality $\norm{f}_\infty \leq \norm{\widehat{f}}_1$.


  For the "$\gtrsim$" part in \eqref{eq:toy}, setting $f_y(x):=\phi_y(y-x)$, we note that $f_y(y)=y \phi_1(0) \approx y$. Again, since $\displaystyle|\phi_1(x)|\lesssim_m  \frac{1}{1+ |x|^{m+2}}$,
  for $y$ large enough we obtain
\begin{eqnarray*}
y^{-2} \norm{f_yW_m}_2^2  \lesssim_m\int _{-\infty}^\infty (1+|x|^{2m}) (1+|y (y-x)|)^{-2m-4} dx\lesssim_m y^{2m-1}.
  \end{eqnarray*}
  Since $\widehat{f}_y$ is supported in $[-2y,2y]$, we have that
  $\norm{\widehat{f}_yW_m}_2 \lesssim_m y^{m + \frac{1}{2}}.$ In conclusion, for $y$ large enough,
   $$\sup_{f} \frac{|f(y)|}{\norm{fW_m}_2+ \norm{\widehat{f}W_m}_2} \gtrsim_m \frac{y}{y^{m + \frac{1}{2}}
   } .$$
   For small $y$ the result is clear.
  \end{proof}
  Note that the constants in
  Theorem 1.3$'$
  depend on $m$.
  Another approach to prove  the "$\lesssim$" part of
\eqref{eq:toy}
  is to use the Landau--Kolmogorov inequality on the half-line, see \cite[Lemma 2.1]{kulikov2023gaussian}. The best constant in this inequality for $d=1$, see \cite{kalyabin}, yields the upper bound
in \eqref{eq:toy} with a constant which behaves  as $\sqrt{m}$.

  As we saw in item {\it (iii)} of Remark \ref{remark:intro},
  in order to obtain Theorem \ref{theorem:mainth}
  the corresponding constant must be independent of $m$.
 A more delicate choice of function $\phi$ is therefore needed.

Before defining the suitable $\phi$, let us first justify how one could arrive at such a choice.
\subsection{Idea of construction    
 of $\phi$}
To begin with, we study the
one-dimensional case with $p=q=2$ and consider $\phi$ in the following form:
  $$\phi(x)=\mathbbm{1}_{[-\delta,\delta]}(x) |x-y|^{2m} P(x) $$ with $\delta>0$  and  a polynomial $P$ to be defined later.

First, since we wish $\norm{(1- \widehat{\phi}(\xi)) |\xi|^{-m}}_2$ to be small, we choose  $P$ to be such that, for small $\xi$, $\displaystyle\widehat{\phi}(\xi)=1+ O(\xi^{\lceil m \rceil+1})$, that is,  
\begin{equation}
\label{eq:condkernel}
    \widehat{\phi}^{(k)}(0)= \int_{-\delta}^\delta  (2 \pi i x)^k P(x) |x-y|^{2m} dx= \begin{cases}
        1, \;\;k=0,\\
        0, \;\;1\leq k \leq \lceil m \rceil.
    \end{cases}
\end{equation}

Second, we observe that
\begin{equation}
\label{eq:condnorm}
\norm{\phi(y-x) |x|^{-m}}_2^2 = \int_{\mathbb{R}} |\phi(x)|^2 |x-y|^{-2m} dx = \int_{-\delta}^\delta |P(x)|^2  |x-y|^{2m} dy.
\end{equation}
Then, defining $$\langle P, Q \rangle= \int_{-\delta}^\delta P(x) Q(x)  |x-y|^{2m} dx,\quad \norm{P} = \langle P,P \rangle ,$$ equations \eqref{eq:condkernel} and \eqref{eq:condnorm} suggest to choose $P$ with the smallest norm $\norm{P}$, for which $\langle P , Q\rangle=Q(0)$ for any polynomial $Q$ of degree $\lceil m \rceil$.
Thus, we take $P$ to be
  the reproducing kernel $K_{\lceil m \rceil}$, that is,  the unique polynomial of degree $\lceil m \rceil$ which satisfies $\langle K_{\lceil m \rceil} , Q\rangle=Q(0)$ for any polynomial $Q$ of degree $\lceil m \rceil$.

\vspace{3mm}
\section{Proofs of
the upper bound in Theorem  \ref{theorem:mainth} and
Theorem \ref{theorem:main2}
}
\label{Section 3}

\subsection{Point evaluation  (Nikolskii-type)  inequalities for algebraic polynomials}
 For a non-negative  integrable function $w$ and $m_0 \in \mathbb{N},$ let $P^*$ be a polynomial of degree $m_0$ which satisfies
 $$\norm{P^*}_{p,w}:=\left(\int_{[-1,1]^d} |P^*(x)|^p w(x) dx\right)^{\frac{1}{p}}
 = \hspace{-7mm} \min_{\tiny \begin{array}{cc}
      & P(0)=1 \\
      & \deg(P)\leq m_0
 \end{array}} \hspace{-5mm}\left(\int_{[-1,1]^d} |P(x)|^p w(x)dx\right)^{\frac{1}{p}}.  $$


Then, it is well known (see, e.g.,    \cite[Chapter 3.10]{devore1993constructive}) that for any polynomial $Q$ of degree $\leq m_0$ the following reproducing formula holds:

\begin{equation}
\label{eq:repro}
   Q(0)= \int_{[-1,1]^d} K_{m_0}(x) Q(x) w(x) dx,
\end{equation}
where
  \begin{equation}
  \label{eq:kern}
      K_{m_0}(x): =K^w_{m_0}(x):=\frac{ |P^*(x)|^{p-2} P^*(x)}{\norm{P^*}_{p,w} ^{p}}.
  \end{equation}
  Here in the case $p=2$ we recover the usual reproducing kernel but, for $p\neq 2$, $K_{m_0}$ is not a polynomial of degree $m_0$.

The main result in this subsection is the following   upper bound for $\norm{K_{m_0}}_{p',w}$:
\begin{theorem}
\label{theorem:niktype}
    Let $1\le p<\infty$, $m\geq 0$, $m_0\in \mathbb{N}$,  and $\alpha \in \mathbb{R}$. Set, for $x \in \mathbb{R}^d$,
\begin{equation}
    \label{w}
w(x):= w_{m,\alpha}(x_1, \dots, x_d)=\left((x_1-\alpha)^2+\sum_{i=2}^d x_i^2 \right)^\frac{pm}{2}.
\end{equation}
Then,
    $$\norm{K_{m_0}}_{p',w} \leq \min\left(C^{m_0+m},C |\alpha|^{-m} (1+ m_0 ^{\frac{d}{p}} +m^{\frac{d}{p}})\right),$$
    where $C:=C(d,p)$.
\end{theorem}
\begin{proof}
Taking into account 
 that
   $$\norm{K_{m_0}}_{p',w} = \frac{1}{\hphantom{_{p,w}}\norm{P^*}_{p,w}}=\sup_{\deg(P)\leq m_0}\frac{|P(0)|}{\hphantom{_{p,w}}\norm{P}_{p,w}},$$
we note that 
 if
for any polynomial $P$ of degree $m_0$
the  Nikolskii-type inequality
\begin{equation}
    \label{eq:nik}
    |P(0)|\leq C^* \norm{P}_{p,w}
\end{equation} holds, one has $\norm{K_{m_0}}_{p',w} \leq C^*$.

    In order to obtain
 $$\norm{K_{m_0}}_{p',w} \leq C(d,p)^{m_0+m},$$ we proceed as follows:

In the one-dimensional case, 
a
combination of the Remez and
Nikolskii
inequalities
(see \cite[Chapter 4, Section 2]{devore1993constructive} and
\cite[Theorem 2.6, p. 102]{devore1993constructive}, respectively) yields
    $$
    \norm{P}_\infty
    \leq C^{m_0}
    \big(\int_{-1}^1 |P(x)|^p dx\big)^{\frac{1}{p}}
    \leq C^{m_0} \big(\int_{1/2}^1 |P(x)|^p dx\big)^{\frac{1}{p}}.
    $$
Iterating this estimate over each variable, we obtain, for any $d\in \mathbb{N}$,
$$\sup_{x \in [-1,1]^d} |P(x)|\leq C^{m_0} \left(\int_{[1/2,1]^d} |P(x)|^p dx \right)^{\frac{1}{p}}.$$
Finally,  
assuming without loss of generality that $\alpha\leq 0$, we conclude that
    $$\sup_{x \in [-1,1]^d} |P(x)|\leq C^{m_0+m} \left(\int_{[-1,1]^d} |P(x) |^p w_{m,\alpha}(x) dx \right)^{\frac{1}{p}}.$$
In order to obtain  the bound
        $$
        \norm{K_{m_0}}_{p',w} \leq C |\alpha|^{-m} (1+ m_0 ^{\frac{d}{p}} +m^{\frac{d}{p}}),
        $$
we make use of the
 Nikolskii inequality
  \begin{equation}
\label{vsp}
F(0) \leq \frac{(2+N)e}{2 \pi} \int_{-1}^1 F(x) dx\end{equation}
 for generalized algebraic polynomials
 defined by
 $$F(z)=|\omega| \prod_{j=1}^k|z-\alpha_j|^{r_j}
 $$
 with $\omega\neq 0$, $r_j\geq 0$, $\alpha_j \in \mathbb{C}$ and generalized degree $\sum_{j=1}^k r_j =N$; see
 \cite[p. 606]{nevai}.
      It is straightforward to obtain
 the  multidimensional version of
 \eqref{vsp}, namely,
    $$F(0, \dots, 0) \leq \left(\frac{(2+N)e}{2 \pi} \right)^d \int_{[-1,1]^d}F(x) dx,$$ where $N$ is the maximum of the generalized degrees of each variable.
    Hence,
   setting $F=|P|^p w$ with $w$ defined by \eqref{w}, we conclude that $$|P(0)|^p w(0) \leq \left(\frac{(2+pm_0+pm)e}{2 \pi} \right) ^d \int_{[-1,1]^d} |P(x)|^p w(x) dx,$$ whence the desired result follows.
\end{proof}

\subsection{Definition and properties of $\phi$}

    \begin{lemma}

    \label{prop:phi}
    Let $1\le p<\infty$, $m\geq 0$ and $\delta,y>0$. Set $\alpha:=y/\delta$.  Let, for $\max(1,\lceil m \rceil)
    \leq m_0
    \in \mathbb{N}
    $ and $x\in \mathbb{R}^d$,
    \begin{equation}
        \label{eq:defphi}
        \phi(x):= \mathbbm{1}_{[-\delta,\delta]^d}(x)  \delta^{-d} K_{m_0}( {x}/{\delta}) w({x}/{\delta}),
     \end{equation} where $K_{m_0}$ and $w:=w_{m,\alpha}$ are given by \eqref{eq:kern} and \eqref{w}, respectively.
    Then
    the following properties hold$:$
\begin{itemize}
\item[\textnormal{(1)}]\; $\displaystyle \left(\int_{\mathbb{R}^d} |\phi(x) |^{p'} w(x/\delta)^{-\frac{p'}{p}} dx \right)^{\frac{1}{p'}}=\delta^{-d/p} \norm{K_{m_0}}_{p',w}
\\
\phantom{space}
\qquad\qquad\qquad\qquad\qquad\;
\leq
\delta^{-d/p} \min\Big(C^{m_0} ,C m_0^{\frac{d}{p}}  (\delta/y)^m\Big);$

\item[\textnormal{(2)}]\; $\displaystyle\displaystyle \widehat{\phi}(0)=\int_{\mathbb{R}^d} \phi(x)dx =1;$
\item[\textnormal{(3)}]\; $\displaystyle |\widehat{\phi}(\xi)| \leq C^{m_0}, \quad \xi \in \mathbb{R}^d;$

\item[\textnormal{(4)}] \; $\displaystyle |1- \widehat{\phi}(\xi)|   \leq   C^{m_0} ( \delta m_0 ^{-1} |\xi|)^{m_0}, \quad \xi \in \mathbb{R}^d;$
    \end{itemize}
with some $C:=C(d,p)$.
    \end{lemma}
        \begin{proof}
            The first item follows from Theorem \ref{theorem:niktype} by a change of variables.
The second statement  of the lemma follows from the reproducing formula \eqref{eq:repro}.

To see (3), we derive from Hölder's inequality and item (1) that
\begin{eqnarray}
    |\widehat{\phi}(\xi)| \leq \norm{\phi}_1 &\leq &\min\Big(C^{m_0} ,C m_0^{\frac{d}{p}}  (\delta/y)^m\Big)\left(\int_{[-1,1]^d} w \right)^{\frac{1}{p}}
    \nonumber\\
    \label{vsp1}
    &\leq& C^{m_0} (1+(\delta /y)^{-m})\min\Big(1 , m_0^{\frac{d}{p}}  (\delta /y)^m\Big)\leq C^{m_0},
\end{eqnarray}
 where we have used
$\displaystyle\Big(\int_{[-1,1]^d} w \Big)^{\frac{1}{p}}\leq     C^m \left( 1+(y/\delta)^{m}\right).
$

For the fourth statement,
 using once again \eqref{eq:repro} and the estimate $\norm{\phi}_1 \leq C^{m_0}$ (see \eqref{vsp1}),
 we arrive at
       \begin{eqnarray*}
 \left|1- \widehat{\phi}(\xi)\right| &=&
 \left|\int_{[-\delta, \delta]^d} \phi(x) \left(\sum_{k=0}^{m_0-1} \frac{( 2 \pi i \langle x, \xi \rangle)^k}{k!}- e^{-2 \pi i \langle x, \xi \rangle}\right) dx\right|
 \\&\leq&  \int_{[-\delta, \delta]^d} |\phi(x)| \frac{| 2 \pi \langle \xi,x \rangle |^{m_0}}{m_0!}  dx\\&\leq&    C^{m_0}\frac{(2 \pi \sqrt{d}\delta |\xi|)^{m_0}}{m_0!}.
  \end{eqnarray*}
The result now follows from the inequality $n!\geq n^n e^{-n}$.
     \end{proof}
  We also need a slight  modification of the previous lemma.\begin{lemma}

    \label{prop:multiphi} Let
    $1\le p<\infty$,
    $m\geq 0$, $\delta>0$, and  $\max(1,\lceil m \rceil)\leq m_0\in \mathbb{N}$. Define the one-dimensional functions \begin{equation}
        \label{eq:multidefphi}
        \phi_1(x):= \mathbbm{1}_{[-\delta,\delta]}(x)  \delta^{-1} K^{w_{m,1}}_{m_0}( {x}/{\delta}) w_{m,1}({x}/{\delta})
     \end{equation} and \begin{equation}
    \phi_0(x):= \mathbbm{1}_{[-\delta,\delta]}(x)  \delta^{-1} K^{w_{0,0}}_{m_0}( {x}/{\delta}) w_{0,0}({x}/{\delta}).
     \end{equation}
  Let now
$$  \phi_d(x_1,\dots, x_d):= \phi_1(x_1) \phi_0(x_2) \cdots \phi_0(x_d).$$ Observe that $\phi_d$ is supported in $[-\delta,\delta]^d$.
    Then,
    the following hold$:$
\begin{itemize}
\item[\textnormal{(1)}]\; $\displaystyle\left(\int_{\mathbb{R}} \phi_1(x_1)^{p'} w_{m,1}({x_1}/{\delta})
^{-\frac{p'}{p}}
 dx_1\right)^{\frac{1}{p'}}\leq C m_0^{\frac{1}{p}} \delta^{-\frac{1}{p}};$

\item[\textnormal{(2)}]\; $\displaystyle\left(\int_{\mathbb{R}} \phi_0(x_2)^{p'} w_{0,0}({x_2}/{\delta})
^{-\frac{p'}{p}}
 dx_2\right)^{\frac{1}{p'}}\leq C m_0^{\frac{1}{p}} \delta^{-{\frac{1}{p}}};$

    \item[\textnormal{(3)}]\; $\displaystyle|\widehat{\phi_d}(\xi)|\leq \norm{\phi_d}_1 \leq C^{m_0},\quad \xi \in \mathbb{R}^d;$
\item[\textnormal{(4)}] \; $\displaystyle |1- \widehat{\phi_d}(\xi)|   \leq   C^{m_0} ( \delta m_0 ^{-1}|\xi|)^{m_0}, \quad \xi \in \mathbb{R}^d;$
   \end{itemize}
with some $C:=C(d,p)$.
    \end{lemma}

With this we conclude the preliminaries and proceed to prove the main results.


\subsection{Proof of  the "$\lesssim$" part of Theorem \ref{theorem:mainth}
}
\label{section:ifpart}

\begin{lemma}
\label{lemma:mainl}
Let $1\le p,q\le \infty$ and $0<A<B$. Then there exists $C:=C(A,B,p,q,d)$ such that, for any $y \in \mathbb{R}^d$,
$$e^{\pi A |y|^2} (1+|y|)^{-\frac{d}{p}}|f(y)|\leq  C \left(\norm{f e^{ \pi A |x|^2}}_p  +\norm{\widehat{f} e^{\pi B  |\xi| ^2}}_q\right).$$
\end{lemma}

\begin{proof}

Let $k$ be a  small positive number and  $N$  a  large integer; both will be defined later.

Without loss of generality, we can assume that $y=(y,0, \dots,0)$ with $y>0$. Let
$k(y)$ be such that
$$k/2<k(y)<k$$ and $$m:=2 \pi A k(y)  y^2 \in \mathbb{N}$$ (for a given $A$, $k(y)$ exists for $y$ large enough). We also
set $$\delta :=k(y)y.$$
Let $\phi_d$ be as in Lemma \ref{prop:multiphi} with these
$\delta$ and $m$ and $$m_0:=Nm.$$

Using the Fourier inversion formula, we have
$$
f(x)=\int_{\mathbb{R}^d} (1- \widehat{\phi_d}(\xi)) \widehat{f}(\xi) e^{2\pi i \langle y, \xi \rangle} d \xi + \int_{\mathbb{R}} \widehat{\phi_d}(\xi) \widehat{f}(\xi) e^{2\pi i\langle y, \xi \rangle} d \xi=:I+II.
$$
In order to estimate $I$, in light of statement  (4) of Lemma \ref{prop:multiphi}, we deduce that, for any $\xi \in \mathbb{R}^d$,
\begin{eqnarray*}
|1- \widehat{\phi_d}(\xi)|   &\leq&    ( C\delta |\xi|)^{Nm}  (Nm)^{-Nm}
\\& \leq &(C \delta)^{Nm} e^{\pi B |\xi|^2} (Nm)^{-\frac{Nm}{2}},
\end{eqnarray*}
 where we used $\displaystyle\frac{(2 \pi B)^{n}|\xi|^{2n}}{n!} \leq e^{2\pi B |\xi|^2}. $

Note that we can find a large enough $N$ independent of $k<1$ such that
\begin{eqnarray*}
&&Nm\left( \log C + \log y + \log k(y) - \frac{1}{2} \log{Nm}\right)
\\
&=&Nm\left( \log C  + \frac{1}{2} \log k(y) -  \frac{1}{2} \log{2\pi A - \frac{1}{2}\log N}\right)\\
&\leq&  -\frac{ \pi A y^2}{3} Nk \log N .
\end{eqnarray*}
Thus,
$$|1- \widehat{\phi_d}(\xi)|\leq \exp\Big({\pi B |\xi|^2 -\frac{  \pi A y^2}{3} Nk \log N}\Big) .
$$
Moreover, in view of  Lemma \ref{prop:multiphi} (3), we obtain
$$|1- \widehat{\phi_d}(\xi)|   \leq C^{Nm}.$$

Next, we apply  Hölder's inequality to get
$$I \leq \norm{\widehat{f}(\xi) e^{\pi B  |\xi| ^2}}_q \left(\int_{\mathbb{R}^d} |1- \widehat{\phi}(\xi)|^{q'} e^{-q'\pi B |\xi|^2}  d \xi\right)^{\frac{1}{q'}}.
$$
Taking into account  our estimates of
$|1- \widehat{\phi_d}(\xi)|$,
we deduce that

\begin{eqnarray*}
&&\left(\int_{\mathbb{R}^d} |1- \widehat{\phi_d}(\xi)|^{q'}e^{-q'\pi B |\xi|^2}  d \xi\right)^{\frac{1}{q'}} \leq \left(\int_{|\xi|\leq y}
\right)^{\frac{1}{q'}} +  \left(\int_{|\xi|\geq y} \right)^{\frac{1}{q'}}
\\
&\leq &  e^{-\frac{ \pi A y^2}{3} Nk \log N} \left(\int_{|\xi|\leq y}  d \xi\right)^{\frac{1}{q'}} +  C^{Nm} \left(\int_{|\xi|\geq y}  e^{-q'\pi B |\xi|^2}  d \xi\right)^{\frac{1}{q'}}
\\
&\leq &
C  y^{\frac{d}{q'}} e^{-\frac{ \pi A y^2}{3} Nk \log N} +C^{Nm} y^{\frac{d-2}{q'}} e^{-\pi B |y|^2},
\end{eqnarray*}
where we have used the simple estimate
$\displaystyle\int_{|x|\geq z} e^{-|x|^2} dx \lesssim_d e^{-z^2} z^{d-2}$ for $z>1$.

Finally, since $$C^{Nm}  e^{-\pi B |y|^2}= \exp\left({-\pi |y|^2( B - 2 Nk(y) A\log C )}\right),$$ for any $B>A$ we can find $N$ and $k$ such that $Nk$ is small enough and $Nk \log N$ is large enough  so that
$$I\leq e^{-\pi A y^2} \norm{\widehat{f}(\xi) e^{B \pi \xi ^2}}_q.$$

For $II$, using Plancherel's formula and Hölder's inequality,
\begin{eqnarray*}
II&=& \int_{[-\delta,\delta]^d} f(x+y) \phi_d(-x) dx \\&\leq& \norm{f(x) e^{\pi A |x|^2}}_p \left(\int_{[-\delta,\delta]^d} e^{-p'\pi A |x+y|^2} |\phi_d(-x)|^{p'}dx\right)^{\frac{1}{p'}}.
\end{eqnarray*}
Here, noting that  $y=(y,0,\dots, 0)$, in light of statement (2) of Lemma \ref{prop:multiphi}, we obtain
$$\left(\int_{[-\delta,\delta]^d} e^{-p'\pi A |x+y|^2} |\phi_d(-x)|^{p'}dx\right)^{\frac{1}{p'}} \leq C m^{\frac{d-1}{p}} \delta^{-\frac{d-1}{p}} \left(\int_{-\delta}^\delta e^{-{p'}\pi A |x+y|^2} |\phi_1(-x)|^{p'
}dx\right)^{\frac{1}{p'}}.
$$
To estimate the last integral,
we take into account Lemma \ref{prop:multiphi}  (1) to get
\begin{eqnarray*}
&&   \left(\int_{-\delta}^\delta e^{-{p'}\pi A |x+y|^2} |\phi_1(-x)|^{p'
}dx\right)^{\frac{1}{p'}} \\
&\leq & \left(\int_{\delta}^\delta  w_{m,1}(-x /\delta)^{-p'/p}|\phi_1(-x)|^{p'}dx\right)^{\frac{1}{p'}} \sup _{x\in [-\delta,\delta]} e^{-\pi A (x+y)^2} w_{m,1}(-x/ \delta)^{\frac{1}{p}}\\
&\leq & C m^{\frac{1}{p}}\delta^{-\frac{1}{p}} e^{-\pi A y^{2}} \sup_{x\in [-1,1]} e^{-2 \pi k(y) A  xy^2} (1+x)^{2 \pi A k(y) y^2 }\\
&\leq & C y^{\frac{1}{p}}e^{-\pi A y^2},
\end{eqnarray*}
   where we have used that $1+x\leq e^x$. Combining the estimate for $I$ and
   $$II \leq
    C y^{\frac{d}{p}}e^{-\pi A y^2}
   \norm{f(x) e^{\pi A  |x|^2}}_p,
   $$
   we arrive at the required result.
 \end{proof}

 The proof of the upper bound in Theorem \ref{theorem:mainth} follows now by a change of variables:
 \begin{proof}[Proof of the "$\lesssim$" part of Theorem \ref{theorem:mainth}]
     Let $A<\alpha<B$ be such that $AB = \alpha^2$. Applying Lemma \ref{lemma:mainl} to $f(\lambda x)$,  we deduce that
     $$C e^{\pi A \lambda ^{-2} |y|^2} (1+|y|)^{-\frac{d}{p}}|f(y)|\lesssim  \norm{f(x) e^{ \lambda ^{-2} A \pi |x|^2}}_p  +\norm{\widehat{f}(\xi) e^{ \lambda ^2 B \pi |\xi| ^2}}_q$$ and the result follows by setting $\lambda^2 = \sqrt{\frac{A}{B}}$.
 \end{proof}
\subsection{Proof of Theorem \ref{theorem:main2}}
Without loss of generality, assume that $y=(y,0, \dots, 0)$ and $y>0$. We proceed as in
Theorem 1.3$'$ 
but this time with $\phi$ from Lemma \ref{prop:phi}.

Consider  $\phi$ given by
    \eqref{eq:defphi}
 with $m_0 = \lceil m+n\rceil $ and $\delta>0$ to be defined later.

First, we claim that

\begin{equation}
        \label{vsp3}
\left(\int_{\mathbb{R}^d} |1- \widehat{\phi}(\xi)|^{q'} |\xi|^{-q'n}  d \xi\right)^{\frac{1}{q'}} \leq C^{m_0} (\delta m_0 ^{-1})^{n-\frac{d}{q'}}=:I.
        \end{equation}
Indeed, set $$\Delta:= \frac\delta {m_0 }
 .$$ Then, using statements (3) and (4) in Lemma \ref{prop:phi}, we deduce that
\begin{eqnarray*}
&&
\left(\int_{\mathbb{R}^d} |1- \widehat{\phi}(\xi)|^{q'} |\xi|^{-q' n}  d \xi\right)^{\frac{1}{q'}}
        \\
 &\leq &C^{m_0} \Delta^{m_0}\left(\int_{|\xi|\leq \Delta^{-1}} |\xi|^{q'(m_0-n)}  d \xi\right)^{\frac{1}{q'}}  + C^{m_0} \left(\int_{|\xi|\geq \Delta^{-1}}  |\xi|^{-q'n}  d \xi\right)^{\frac{1}{q'}}  \\
 &\leq & C^{m_0} \Delta^{n-\frac{d}{q'}},
\end{eqnarray*}
completing the proof of \eqref{vsp3}.

 Second, recalling that $\alpha=y/\delta,$ we note that $ |x-y|^{-m} =  w(x/\delta)^{-1/p} \delta^{-m}$, cf. \eqref{w}. Thus, we deduce from Lemma \ref{prop:phi} (1) that
 $$\left(\int_{\mathbb{R}^d}  |x-y|^{-p'm} |\phi(x)|^{p'} dx\right)^{\frac{1}{p'}}\leq \delta^{-d/p-m} \min\left(C^m ,C m^{\frac{d}{p}}  (\delta/y)^m\right)=:II.$$

Finally, it remains to choose  $\delta$. We proceed as follows.
In order to prove inequality \eqref{eq:m}, set $$\delta:=\sqrt{m_0}.$$
Then, using that $m_0 \approx_{p,q,\varepsilon ,d} n \approx_{p,q,\varepsilon ,d} m$, we obtain
$$I+II\leq C^m m^{-m/2}.$$
Thus, we deduce that
\begin{equation}
\label{ineq:tempo}
    |f(y)|\leq C^m m^{-m/2}
\left(\norm{f(x)|x|^m}_p+  \norm{\widehat{f}(\xi)|\xi|^n}_{q}\right),
\end{equation} whence the result follows by observing that $m!\leq m^m$.
This proves  \eqref{eq:m}.

 To derive inequality
 \eqref{eq:y2}, set $$\delta:=\lambda \frac{m}{y}$$
 with $\lambda>0$ to be defined later. Then, in light of $n-\frac{d}{q'}=m - \frac{d}{p}$, we derive
\begin{equation}
    \label{lambda1}
  I\leq y^{-m + \frac{d}{p}} C^m \lambda^{m-\frac{d}{p}}.\end{equation}
Hence, making $\lambda$ small enough, we obtain
\begin{equation}
    \label{lambda2}
  I\leq y^{-m+ \frac{d}{p}}.
  \end{equation}
  Moreover, for our choice of $\delta$ we have $$II\leq C m^{\frac{d}{p}} \delta^{-\frac{d}{p}}  y^{-m} \leq C y^{-m+\frac{d}{p}}.$$

In conclusion,

$$
|f(y)| \leq C y^{-m+\frac{d}{p}}
 \left(\norm{f(x)|x|^m}_p+  \norm{\widehat{f}(\xi)|\xi|^n}_q\right).
$$
It remains to note that, for $|y|<1$, by \eqref{ineq:tempo},
$$|f(y)| \leq  C^m m^{-m/2} \left(\norm{f(x)|x|^m}_p+  \norm{\widehat{f}(\xi)|\xi|^n}_q\right) \leq  C\left(\norm{f(x)|x|^m}_p+  \norm{\widehat{f}(\xi)|\xi|^n}_q\right).$$
This completes the proof of inequality \eqref{eq:y2}.

\hfill
\qedsymbol{}


\vspace{3mm}

 \section{Proof of the lower bound 
  in Theorem \ref{theorem:mainth}}
\label{Section 4}

\subsection{
The case
 $\alpha=\frac{1}{\sqrt{2}}$
}
Before presenting the proof of the lower bound in \eqref{eq:mainequation} by means of a specific construction, let us show how, for a special choice of $\alpha$, such a construction arises in a natural way.

Our starting point is the following observation:
\begin{proposition}\label{proposition}
Let $\alpha=\frac{1}{\sqrt{2}}$. Then, for $y\in\mathbb{R},$
\begin{equation}
\label{eq:derivatives--}
\sup
\frac{|g(y)|}{N(g)}e^{- \pi \alpha y^2} \leq \sup_{f\in L_1+L_2}\frac{|f(y)|}{\norm{f(x) e^{\pi \alpha |x|^2}}_2 + \norm{\widehat{f}(\xi) e^{\pi \alpha |\xi|^2}}_2},
\end{equation}
where the supremum
on the left-hand side
 is taken over all
bandlimited functions $g\in L_2(\mathbb{R})$ and


\begin{equation}
\label{eq:derivatives}
N(g)^2:= \norm{g}_2^2 + 2^{\frac{1}{4}}\sum_{n=0}^\infty \frac{ 2^{\frac{n}{2}} |g^{(n)}(0)|^2}{(2\pi)^n n!}.
\end{equation}
\end{proposition}
\begin{proof}

Let $g\in L_2(\mathbb{R})$ be
such that $\widehat{g}$ is compactly supported
 and set $f(x):=e^{- \pi \alpha x^2} g(x)$. Then we have
$$\widehat{f}(\xi)=\alpha^{-\frac{1}{2}} \int_{\mathbb{R}} \widehat{g}(\xi- \eta) e^{-\pi \alpha^{-1} \eta^2} d \eta$$ and
$$\norm{f e^{\pi \alpha x^2}}_2 = \norm{g}_2.$$
Moreover, it is clear that
$$\norm{\widehat{f} e^{\pi \alpha \xi^2}}_2^2
 = \alpha^{-1}\int_{\mathbb{R}} \left(\int_{\mathbb{R}} \widehat{g}(\xi-  \eta) e^{-\pi \alpha^{-1}  \eta^2} d \eta\right) \left(\int_{\mathbb{R}} \overline{\widehat{g}(\xi-  \nu)} e^{-\pi \alpha^{-1}  \nu^2} d \nu\right)  e^{2\pi \alpha \xi^2} d \xi.
 $$
  Further, we observe that by setting
 $$\left\{
  \begin{array}{ll}
   u:= \xi -  \eta,  \\
   v:=\xi- \nu,  \\
   w:=\xi,
  \end{array}
\right.
$$
 we arrive at
$$2 \pi \alpha \xi^2 - \pi \alpha^{-1}( \eta^2+  \nu^2)= -2\pi w^2(\alpha^{-1}- \alpha)+2 \pi \alpha^{-1}w (u+v) - \pi \alpha^{-1}(u^2+v^2).$$ Thus, by a change of variables and Fubini's theorem,
\begin{eqnarray*}
        \norm{\widehat{f} e^{\pi \alpha \xi^2}}_2^2
        &=&
        \alpha^{-1}\int_{\mathbb{R}^3}  \widehat{g}(\xi-  \eta) e^{-\pi \alpha^{-1}  \eta^2}   \overline{\widehat{g}(\xi-  \nu)} e^{-\pi \alpha^{-1}  \nu^2}  e^{2\pi \alpha \xi^2}d \eta d \nu d \xi \\
       &=& \alpha^{-1}\int_{\mathbb{R}^2}  \widehat{g}(u)    \overline{\widehat{g}(v)} e^{- \pi \alpha^{-1}(u^2+v^2)}\left(\int_{\mathbb{R}} e^{-2\pi w^2(\alpha^{-1}- \alpha)+2 \pi \alpha^{-1}w (u+v)}  d w \right) du dv.
\end{eqnarray*}
Taking into account that  $\alpha<1$, one has 
$$\int_{\mathbb{R}} e^{-2\pi w^2(\alpha^{-1}- \alpha)+2 \pi \alpha^{-1}w (u+v)}  d w = \left(2 (\alpha^{-1}- \alpha)\right)^{-\frac{1}{2}} e^{ \frac{\pi (u+v)^2}{2 \alpha^2 (\alpha^{-1}-\alpha)}}.$$
Therefore, since $\alpha= \frac{1}{\sqrt{2}}$, after an application of Fubini's theorem we conclude that
\begin{eqnarray*}
     \norm{\widehat{f} e^{\pi \alpha \xi^2}}_2^2
     &=&
     \left(2 \alpha^2 (\alpha^{-1}- \alpha)\right)^{-\frac{1}{2}} \int_{\mathbb{R}^2}  \widehat{g}(u)    \overline{\widehat{g}(v)} e^{- \pi \alpha^{-1}(u^2+v^2)} e^{ \frac{\pi (u+v)^2}{2 \alpha^2 (\alpha^{-1}-\alpha)}} du dv\\
     &=& 2^{\frac{1}{4}}\int_{\mathbb{R}^2}  \widehat{g}(u)    \overline{\widehat{g}(v)} e^{2  \pi \sqrt{2}uv} du dv
     =
     2^{\frac{1}{4}} \sum_{n=0}^\infty \frac{ 2^{\frac{n}{2}} |g^{(n)}(0)|^2}{(2\pi)^n n!}.
 \end{eqnarray*}
The proof
is now complete.

\end{proof}

\begin{remark}
   {\textnormal{
The supremum on the left-hand side of
\eqref{eq:derivatives--} can be calculated explicitly, see Section 5. }}
\end{remark}

In light of Proposition \ref{proposition}, a reasonable approach to estimate
the supremum on the right-hand side of
\eqref{eq:derivatives--}
from below is the following:
for each $y>0$, set  $f(x):=e^{-\pi \alpha x^2}g(x)$ with
$g$ of minimal $\norm{g}_2$ among those $g$ with $\widehat{g}$ supported on $[-M/2,M/2]$ such that $g(y)=1$ and $g^{(n)}(0)=0$ for $0\leq n \leq  N $, where $M$ and $N$ are parameters to be optimized.

By taking the Fourier transform, this extremal  problem can be equivalently stated as follows: find $\widehat{g}$ supported on $[-M/2,M/2]$ with minimal $\norm{\widehat{g}}_2$ such that
$$\int_{-\frac{M}{2}}^{\frac{M}{2}} \widehat{g}(\xi) e ^{2 \pi i\xi y} d\xi =1\quad\mbox{ and}\quad \int_{-\frac{M}{2}}^{\frac{M}{2}} \widehat{g}(\xi) \xi^n d\xi =0\quad\mbox{ for}\; 0\leq n\leq N.$$ It is well known that the solution to this problem is given by $\widehat{g}=h$ with
\begin{equation}
\label{eq:defh}
    h(\xi)= \frac{1}{\widetilde{E}_N (e^{-2 \pi i \xi y})^{2}}\mathbbm{1}_{[-\frac{M}{2}, \frac{M}{2}]}(\xi) \left(e^{-2 \pi i \xi y}-P^*(\xi) \right),
\end{equation}
where
$P^*$ is the orthogonal projection of $e^{-2 \pi i \xi y}$ onto the space of algebraic polynomials of degree $N$ and $\widetilde{E}_N(e^{-2 \pi i \xi y})$
is the best  $L_2$-approximation of $e^{-2 \pi i \xi y}$ by polynomials, that is,
\begin{eqnarray*}
    \widetilde{E}_N(e^{-2 \pi i \xi y})&:=&\inf_{P \text{ of degree}\leq N}\Big( \int_{-\frac{M}{2}}^{\frac{M}{2}} \left|e^{-2 \pi i \xi y}-P(\xi) \right|^2 d \xi
\Big)^\frac12\\&=&\Big(\int_{-\frac{M}{2}}^{\frac{M}{2}} \left|e^{-2 \pi i \xi y}-P^*(\xi) \right|^2 d \xi\Big)^\frac12.
\end{eqnarray*}
 It is thus clear that the minimal value is $\displaystyle\norm{\widehat{g}}_2 = {1}/{\widetilde{E}_N(e^{-2 \pi i \xi y})}.$

\subsection{Preliminaries on approximation theory}
We now proceed to elaborate on the ideas presented in the previous subsection. From now on, we replace $\exp({-2 \pi i \xi y})$ with $\cos({2 \pi  \xi y})$, as it will simplify the following arguments. (Equivalently, we  replace  $f(y)$ in equation \eqref{eq:mainequation} with $\frac{f(y)+f(-y)}{2}$.)

Our first aim is to estimate $\widetilde{E}_N\left(\cos({2 \pi  \xi y})\right)$. Note that by a change of variables it suffices to consider $$E_N\left(\cos({2 \pi  \xi  D})\right)=\inf_{P \text{ of degree}\leq N}\Big( \int_{-\frac{1}{2 \pi}}^{\frac{1}{2 \pi }} \left|\cos({2 \pi  \xi  D})-P(\xi) \right|^2 d \xi\Big)^\frac12, \quad D \in \mathbb{R}.$$

We will now  show that
the cosine function
 with frequency $D$ cannot be meaningfully approximated by polynomials of degree much smaller than its frequency.
 To prove this,
we use the close  relationship between
the best error of approximation of a function and  its smoothness (see, e.g., Chapter 7 in \cite{devore1993constructive}).

\begin{lemma}
\label{lemma:Ecd}
 There exists $c>0$ small enough such that,  for any $D\geq 1$,

    $$1\approx E_{\lfloor cD\rfloor}\left(\cos(2 \pi \xi D)\right).$$

\end{lemma}
\begin{proof}

    Set $g_D(\xi):=\cos(2 \pi  \xi D).$ We may clearly assume that $D$ is large enough.
    Then, setting
    $s:=\frac{1}{2D}$ and $I:=[-\frac{1}{4\pi},\frac{1}{4 \pi}]$, we see that
    $$
    \left(\int_{I} |g_D(x+s)-g_D(x)|^2 \mathbbm{1}_{I}(x+s) dx \right)^{\frac{1}{2}}  \approx   |I \cap (I-s)|^{\frac{1}{2}} \approx 1.
    $$
    Defining the $L_2$-modulus of continuity by
    $$\omega(f,t,I):= \sup_{|s|\leq t} \left(\int_{I} |f(x+s)-f(x)|^2 \mathbbm{1}_{I}(x+s) dx \right)^{\frac{1}{2}},
    $$
    we note  that
    $$\omega(g_D,t ,I)\gtrsim 1
\quad\mbox{    for}\quad t \geq \frac{1}{2D}.$$

    Next, let $n$ be an integer satisfying $\frac{D}{3}\leq  n\leq \frac{D}{2}$. An application of the inverse approximation theorem {\cite[p. 221]{devore1993constructive}} yields that there exists $K>0$ such that, for
    any positive $c<1/3$,
    \begin{eqnarray*}
            1 \lesssim   \omega(g_D,n^{-1}, I) \leq \frac{K}{n} \left(\sum_{k=0}^n E_k(g_D)\right) &\leq& \frac{K}{n} \left(\sum_{k=0}^{\lfloor{3cn}\rfloor} E_k(g_D)\right) + \frac{K}{n} \left(\sum_{k={\lfloor{3cn}\rfloor}+1}^{n} E_k(g_D)\right) \\
            &\leq & 3cK + K E_{\lfloor{3cn}\rfloor+1}(g_D),
    \end{eqnarray*}
    where we have used that
    $E_k(g_D) \leq \norm{g_D}_2\leq 1$.
    Therefore, for $c$ small enough, we conclude that $$1\lesssim E_{\lfloor{3cn}\rfloor+1}(g_D) \leq E_{\lfloor cD\rfloor}(g_D)  \leq 1.$$

   \end{proof}
For $L_p$ spaces with $p\ne2$,
the construction given in \eqref{eq:defh}
needs to be slightly modified. Roughly speaking, we need to perturb $\cos(2 \pi \xi y)$ in order to make it orthogonal to every polynomial up to a certain degree while also keeping control of the $L_2$ and $L_p$ norms of its Fourier transform. In Lemma \ref{lemma:legbess} we collect the results required to show that the modified $h$ given  in \eqref{eq:defh2} below
 still has the desired properties.
\begin{lemma}
\label{lemma:legbess}
   For $m\in \mathbb{N}$, we define the $m$-th Legendre polynomial
   $$P_m(x)=\frac{1
    }{2^m m!} \left[(x^2-1)^m \right]^{(m)};$$
   and the normalized $m$-th Legendre polynomial
    $$\widetilde{P}_m(x)= i^m \sqrt{(2n+1)\pi }P_m(2 \pi x).$$
    Then, the following hold:

\begin{itemize}
\item[\textnormal{(1)}]\; The sets $\{\widetilde{P}_m\}_{m\geq 0}\subset L_2\left([-\frac{1}{2 \pi}, \frac{1}{2 \pi}]\right)$ and $\{\widetilde{P}_{2k}\}_{k\geq 0}\subset L_{2,\text{even}}\left([-\frac{1}{2 \pi}, \frac{1}{2 \pi}]\right)$ are orthonormal systems in the corresponding spaces.

\item[\textnormal{(2)}]\; The Fourier transform of the normalized Legendre polynomials is given by
\begin{equation}
\label{eq:fourlleg}
\int_{-\frac{1}{2 \pi}}^\frac{1}{2 \pi} e^{-2 \pi i \xi x} \widetilde{P}_{2k}(x) dx = \sqrt{\frac{ (4k+1)}{2|\xi| }} J_{2k+1/2}( |\xi|)=: j_k(\xi),
\end{equation}
where the Bessel function $J_\alpha$ is defined by $$ J_\alpha(x) = \sum_{m=0}^\infty \frac{(-1)^m}{m!\, \Gamma(m+\alpha+1)} {\left(\frac{x}{2}\right)}^{2m + \alpha}. $$

 \item[\textnormal{(3)}]\;
For even $f:\mathbb{R} \to \mathbb{C}$ and $g: [-\frac{1}{2 \pi}, \frac{1}{2 \pi}] \to \mathbb{C}$, we define their corresponding Fourier--Bessel and Fourier--Legendre coefficients
$$c_k(f)=\int_{-\infty} ^\infty f(x) j_k(x)  dx \quad \mbox{and}\quad d_k(g)=\int_{-\frac{1}{2 \pi}} ^{\frac{1}{2 \pi}} g(x) \tilde{P}_{2k}(x)  dx,$$ as well as their partial sums 
$$S^{J}_n(f) (x)=\sum_{k=0}^n c_k(f) j_k(x), $$
        $$S^{L}_n(g) (x)=\sum_{k=0}^n d_k(g) \tilde{P}_{2k}(x), $$
 and Riesz means   $$R^{X}_n(h)(x)=\frac{\lambda_0 S^X_0 + \cdots + \lambda _n S_n^X}{\lambda_0 + \cdots + \lambda _n},$$ where $\lambda_k =4k+3$ and $(X,h)\in \{(J,f),(L,g)\}$.
    \\
    Finally, set $\Lambda_n= \lambda_1+ \cdots + \lambda_n$ and let $$V_n^X(h)=\frac{ \Lambda_{2n} R^X(h)_{2n} - \Lambda_n R^X(h)_n}{\Lambda_{2n}- \Lambda_n}= \frac{\lambda_{n+1} S^X_{n+1}(h)+ \cdots + \lambda_{2n} S^X_{2n}(h) }{\Lambda_{2n}- \Lambda_n}.$$ Then,
    \begin{equation}
    \label{eq:VjVl}
        V_n^J(\widehat{g})= \widehat{V_n^L(g)}
    \end{equation} and
    \begin{equation}
    \label{eq:boundVn}
        \norm{V_n^J(f)}_p\leq C(p) \norm{f}_p,\quad 1<p<\infty.
    \end{equation}

   \end{itemize}

    \end{lemma}
        \begin{proof}
Item           (1) follows from
\cite[\href{ https://dlmf.nist.gov/14.17.E6}{(14.17.6)}]{NIST:DLMF}, together wih the facts that
 $\tilde{P}_{2k+1}$ is odd and $\tilde{P}_{2k}$ is even.

        Relation \eqref{eq:fourlleg} in statement  (2) follows by a change of variable from \cite[\href{https://dlmf.nist.gov/18.17.E19}{(18.17.19)}]{NIST:DLMF}.

Formula \eqref{eq:VjVl} is a straightforward consequence of \eqref{eq:fourlleg}. Finally, to prove inequality \eqref{eq:boundVn}, taking $\alpha = -1/2$ and $\gamma =0$ in Theorem 1 of \cite{ciaurri}, we deduce that,
    for $1<p<\infty$ and even $f$,
   $$\norm{R^J_n(f)}_p \leq C(p) \norm{f}_p.$$ Thus, using that $\Lambda_n = n^2 + 5n+3$, we also have
   $$\norm{V^J_n(f)}_p \leq C(p)\frac{\Lambda_{2n} +\Lambda_{n}}{\Lambda_{2n}- \Lambda_n}\norm{f}_p \leq C'(p)\norm{f}_p.$$
\end{proof}

\subsection{Construction of extremizers and
proof of  the "$\gtrsim$" part of Theorem \ref{theorem:mainth}
}\label{4.3}


We are now in a position to construct
the extremal (up to  constants)  function in \eqref{eq:mainequation}.
Let
$$g_D(\xi):= \mathbbm{1}_{[-\frac{1}{2 \pi }, \frac{1}{2 \pi }]}(\xi)\cos\left(2 \pi  \xi D\right), \quad\xi\in\mathbb{R},
$$
and
\begin{equation}
\label{eq:defh2}
    h(\xi):=  \mathbbm{1}_{[-\frac{1}{2 \pi }, \frac{1}{2 \pi }]}(\xi) \left( g_D(\xi)- V^L_{N}(g_D)(\xi)\right),\quad\xi\in\mathbb{R}.
\end{equation}

\begin{lemma} Let $1<p<\infty$ and $D\geq 1$.
\label{lemma:bestapprox}
 Assume that
 $4N\leq \lfloor cD \rfloor$ with $c$ as in Lemma \ref{lemma:Ecd}.  Then
\begin{itemize}
\item[\textnormal{(1)}]\; $\widehat{h}(D)\approx 1;$
\item[\textnormal{(2)}]\; $\norm{\widehat{h}}_p \leq C(p);$
\item[\textnormal{(3)}]\; $\widehat{h}^{(n)}(0)=0$ for $0\leq n \leq 2N,$ equivalently$,$
    $$\langle h, Q \rangle :=\int_{-\frac{1}{2 \pi}}^{{\frac{1}{2 \pi}}} h(\xi) Q(\xi) d \xi =0$$ for any polynomial $Q$ of degree less than or equal to  $ 2N $.
\end{itemize}

\end{lemma}
\begin{proof}
In order to obtain  (1), we note that $V_N^L(g_D)=\sum_{j=0}^{2N} c_j^N \tilde{P}_{2j}  \langle \tilde{P}_{2j}, g_D\rangle$ with $0\leq c_j^N \leq 1$. Then,
\begin{eqnarray*}
        \widehat{h}(D)&=&\langle g_D , g_D - V_N^L(g_D) \rangle = \langle g_D, g_D \rangle  - \sum_{j=0}^{2N} c_j^N|\langle \tilde{P}_{2j}, g_D\rangle|^2\\
        &\geq& \langle g_D, g_D \rangle- \sum_{j=0}^{2N}  |\langle \tilde{P}_{2j}, g_D\rangle|^2  \geq E_{4N}(g_D)^2 \approx 1,
 \end{eqnarray*}
   where  in the last estimate we have used Lemma \ref{lemma:Ecd}.

For (2), we observe that $\norm{\widehat{g}_D}_p\leq C(p)$ and, in view of relation \eqref{eq:VjVl}, $\widehat{h}=\widehat{g}_D -V_n^J(\widehat{g}_D)$; so the result follows from inequality \eqref{eq:boundVn}.

   Finally, to see that (3) holds, we observe that, for any $0\leq k \leq N$,
   $$\langle V_N^L(g_D), \widetilde{P}_{2k}\rangle =\langle S_{N+1}^L (g_D), \widetilde{P}_{2k}\rangle = \langle g_D, \widetilde{P}_{2k}\rangle.$$ Thus, $h$ is orthogonal to any even polynomial of degree $\leq 2N$ and, since $h$ is even, we conclude that it is also orthogonal to any odd polynomial.

\end{proof}

   We are finally in a position to prove the "$\gtrsim$" part of relation \eqref{eq:mainequation}. We divide the proof into two parts: $d=1$ and $d>1$.

\begin{proof}[Proof of the lower bound in Theorem \ref{theorem:mainth} for $d=1$]
    For a positive large enough $y$, set $$M:= \lambda y\quad \mbox{and} \quad D:=M y=\lambda y^2,$$ with $\lambda>0$ being a  small constant depending on $\alpha$ in a way to be defined later. Since $y$ is large enough, there exists an integer $N$ such that $$\frac{1}{8} c yM  <N<\frac{1}{4} \lfloor c yM\rfloor,$$
with $c$ given in Lemma \ref{lemma:Ecd}.

For $D $ and $N$ as above and $h$ given in \eqref{eq:defh2}, set $$h_M(\xi) :=M^{-1}h(\xi/M)$$ and
$$f(x):=\widehat{h_M}(x) e^{-\pi \alpha x^2}.$$
It follows from Lemma \ref{lemma:bestapprox} that $$\norm{f e^{\pi \alpha x^2}}_p = \norm{\widehat{h_M}}_p=M^{-\frac{1}{p}} \norm{\widehat{h}}_p \lesssim_p M^{-\frac{1}{p}}$$ and  $$e^{ \pi \alpha y^2}f(y) =\widehat{h_M}(y)= \widehat{h}(D) \approx 1.$$
All that remains is to bound $\norm{\widehat{f}(\xi) e^{\pi \alpha \xi^2}}_q$ from above.

Before proceeding with the estimate of $\widehat{f}$, we define $Q_{2N}$ to be the Taylor expansion of degree $2N$ of $e^{-\pi \alpha^{-1}x^2}$ at $0$. Then,
for any $x \in \mathbb{R}$, $$|Q_{2N}(x)-e^{-\pi \alpha^{-1}x^2}|\leq  \frac{(\pi \alpha^{-1} x^2)^{N+1}}{(N+1)!}.$$

Let $\mu$ be a large number to be chosen later and assume that $|\xi|\leq \mu M$.
Using that $h$ is orthogonal to any polynomial of degree $2N$,  we have
\begin{eqnarray}
\nonumber
|\widehat{f}(\xi)|&=&\alpha^{-\frac{1}{2}}
\Big|\int_{-\frac{M}{2\pi}}^{\frac{M}{2\pi}} h_M(\eta) e^{-\pi \alpha^{-1} (\xi - \eta)^2} d \eta\Big|\\
\nonumber
&=& \alpha^{-\frac{1}{2}}
\Big|\int_{-\frac{M}{2\pi}}^{\frac{M}{2\pi}} h_M(\eta) \left(e^{-\pi \alpha^{-1} (\xi - \eta)^2} -Q_{2N}(\xi - \eta)\right) d \eta\Big|
\\
\nonumber
    &\lesssim_\alpha &M^{\frac{1}{2}} \norm{h_M}_2 \sup_{|\eta| \leq \mu M+ M} \big|e^{- \pi \alpha^{-1} \eta^2} -Q_{2N}(\eta)\big|\\
    \label{eq:lowfreq2}
    &\lesssim_\alpha &M^{\frac{1}{2}} \norm{h_M}_2 \frac{(\pi \alpha^{-1}(\mu +1)^2 M^2 )^{N+1}}{(N+1)!}.
\end{eqnarray}
For $|\xi|> \mu M$, by monotonicity, we derive that
\begin{eqnarray}
\label{eq:highfreq2}
    |\widehat{f}(\xi)|
    &=&\alpha^{-\frac{1}{2}}
    \Big|\int_{-\frac{M}{2\pi}}^{\frac{M}{2\pi}} {h}_M(\eta) e^{-\pi \alpha^{-1} (\xi - \eta)^2} d \eta \Big|\nonumber\\
    &\lesssim_\alpha& e^{-\pi \alpha^{-1} (|\xi| - \frac{M}{2 \pi})^2} M^{\frac{1}{2}} \norm{h_M}_2.
   \end{eqnarray}

With the previous bounds in hand, we now proceed to estimate
$$\norm{\widehat{f} e^{\pi \alpha \xi^2}}_q.$$
First, using  \eqref{eq:highfreq2}, we deduce that
\begin{eqnarray*}
I&:=&\int_{\mu M}^\infty e^{q \pi \alpha \xi^2} |\widehat{f}(\xi)|^q d \xi \\
&\lesssim_\alpha & M^{\frac{q}{2}}\norm{h_M}_2^q\int_{\mu M}^\infty e^{q \pi \alpha \xi^2} e^{-q\pi \alpha \left(\frac{M/(2 \pi)-\xi}{\alpha}\right)^2} d \xi.
\end{eqnarray*}
Since $\alpha<1$, there exists a  large enough $\mu$ depending only on $\alpha$ such that 
$$\xi^2 - \left(\frac{M/(2 \pi)-\xi}{\alpha}\right)^2 \leq \frac{1- \alpha^{-2}}{2} \xi^2 $$ for $\xi \geq \mu M$ and for any $M$. Fix this $\mu$.
    Then for such $\mu$ and any $M$,
    \begin{eqnarray*}
    I &\lesssim_\alpha& M^{\frac{q}{2}} \norm{h_M}_2^q \int_{\mu M}^\infty  e^{\frac{q\pi}{2} \alpha (1- \alpha^{-2}) \xi ^2} d \xi  \\&\lesssim_{\alpha,q}&  M^{\frac{q}{2}-1} \norm{h_M}_2 ^q e^{\frac{q\pi}{2}  \alpha (1-\alpha^{-2}) \mu ^2 M^2},
    \end{eqnarray*}
    where we have used that  $\displaystyle\int_y ^\infty e^{-u^2} du \lesssim y^{-1} e^{-y^2}$. 

Second, in light of  inequality \eqref{eq:lowfreq2},
applying now 
 the estimate $\displaystyle\int_0 ^y e^{u^2} du \lesssim y^{-1} e^{y^2}$, we obtain


\begin{eqnarray*}II&:=&\int_0^{\mu M} e^{q \pi \alpha \xi^2} |\widehat{f}(\xi)|^q  d \xi \\&\lesssim_{\alpha,q}
& M^{\frac{q}{2}-1} \norm{h_M}_2^q \left(e^{ \pi \alpha \mu ^2 M^2 } \frac{(\pi \alpha^{-1}(\mu +1)^2M^2)^{N+1}}{(N+1)!}\right)^q.
\end{eqnarray*}

Here we recall that $N> \frac{c}{8\lambda} M^2$ with $\lambda$ to be chosen. Thus, using the Stirling formula, we have, with a constant $C:=C(\alpha)$,
$$
\frac{(\pi \alpha^{-1}(\mu +1)^2M^2)^{N+1}}{(N+1)!} \leq \frac{(CM)^{2N+2}}{\left( \lambda ^{-1} M^2\right)^{N+1}}= (C^2\lambda)^{N+1} \leq e^{\frac{c}{8 \lambda} M^2 \log{C^2\lambda}}.$$
    Therefore, for $\lambda:=\lambda(\alpha)$ small enough and for some $K:=K(\alpha,q)>0$,
    $$I+II\lesssim_{\alpha,q}  e^{-KM^2} .$$
Proceeding analogously for $\displaystyle\int_{-\infty}^0 |\widehat{f}(\xi)|^q e^{q \pi \alpha \xi^2} d \xi$,  we deduce that, for $y$ large enough,
$$\norm{\widehat{f} e^{\pi \alpha \xi^2}}_q\lesssim_{\alpha,q}  e^{-KM^2},$$
completing the estimate of
 $\norm{\widehat{f} e^{\pi \alpha \xi^2}}_q$ from above.

Finally, recalling that $M \approx_\alpha y$, we conclude that, for $y$ large enough and  some positive $K':=K'(\alpha,q)$,
$$\sup_{f\in L_1+L_2}\frac{|f(y)|}{\norm{f(x) e^{\pi \alpha |x|^2}}_p + \norm{\widehat{f}(\xi) e^{\pi \alpha |\xi|^2}}_q} \gtrsim_{\alpha,p,q} \frac{e^{-\pi \alpha y^2}}{y^{-\frac{1}{p}}+e^{-K'y^2}} \approx_{\alpha,p,q} (1+|y|)^{\frac{1}{p}} e^{-\pi \alpha y^2}. $$

This completes the proof of   Theorem \ref{theorem:mainth}  for $d=1$.
\end{proof}

    The multidimensional result  follows directly from
the one-dimensional case.
    Indeed,

\begin{proof}[Proof of the lower bound in Theorem \ref{theorem:mainth} for $d>1$]
Let $y=(y_1, \dots, y_d) \in \mathbb{R}^d$. Since the quantity $$\sup_{f\in L_1(\mathbb{R}^d)+L_2 (\mathbb{R}^d)}\frac{|f(y)|}{\norm{f(x) e^{\pi \alpha |x|^2}}_p + \norm{\widehat{f}(\xi) e^{\pi \alpha |\xi|^2}}_q}$$ only depends on $|y|$, we can assume without loss of generality that $y_1=\cdots y_d= {|y|}/{\sqrt{d}}$.
   Let $f: \mathbb{R} \to \mathbb{C}$ and define $g(x_1, \dots, x_d)=f(x_1) \cdots f(x_d)$. Then we also have $\widehat{g}(\xi_1, \dots, \xi_d)=\widehat{f}(\xi_1) \cdots \widehat{f}(\xi_d)$. Clearly,
    $$\norm{g(y) e^{\pi \alpha |y|^2}}_p = \norm{f(y) e^{\pi \alpha y^2}}_p^{d}\quad\mbox{ and}\quad  \norm{\widehat{g}(\xi) e^{\pi \alpha |\xi|^2}}_q = \norm{\widehat{f}(\xi) e^{\pi \alpha \xi^2}}_q^{d}.$$
    Hence, from the one-dimensional lower bound we deduce that \begin{eqnarray*}
            \sup_{f\in L_1(\mathbb{R}^d)+L_2 (\mathbb{R}^d)}\frac{|f(y)|}{\norm{f(x) e^{\pi \alpha |x|^2}}_p + \norm{\widehat{f}(\xi) e^{\pi \alpha |\xi|^2}}_q} &\geq &\\
    \sup_{f\in L_1(\mathbb{R})+L_2 (\mathbb{R})}\Big(\frac{|f(|y|/\sqrt{d})|}{\norm{f(x) e^{\pi \alpha |x|^2}}_p + \norm{\widehat{f}(\xi) e^{\pi \alpha |\xi|^2}}_q}\Big)^d  
    &\gtrsim& \hspace{-3mm}_{\alpha,d,p,q} \; e^{-\pi \alpha |y|^2} (1 + |y|)^{\frac{d}{p}}.
    \end{eqnarray*}     This completes the proof of Theorem \ref{theorem:mainth}.
\end{proof}

\vspace{3mm}
\section{Final remarks and open problems}
\label{Section 5}

Throughout this section we restrict ourselves  to the case $p=q=2$, unless otherwise stated.

   1.
It is natural to expect that additional assumptions on a function
$f$ satisfying
\begin{equation}\label{vsp5}
\norm{f(x) e^{\pi \alpha |x|^2}}_2 + \norm{\widehat{f}(\xi) e^{\pi \alpha |\xi|^2}}_2 <\infty
\end{equation}
with $0<\alpha<1$
imply a stronger decay rate at infinity than 
 in the general case.
In \cite{kulikov2023gaussian},
the authors proved that
 \eqref{vsp5} with $0<\alpha<1$
 implies that
$|f(x)|\lesssim e^{-\pi \alpha |x|^2}$ for any $x\in \mathbb{R}$
 in the case when
$f(x)=\mathcal{L}\mu(\pi|x|^2)$, where
$\displaystyle\mathcal{L}\mu(x)$ is the Laplace transform
of
 a finite measure $\mu$ with support on the positive real line, i.e.,
 $\mathcal{L}\mu(x)=\int_0^\infty e^{-xt} d\mu(t)$.

The following result
is the counterpart of Theorem \ref{theorem:mainth}
for radial functions:




\begin{theorem}\label{radial}
Let $0<\alpha<1$. Then, for $y\in \mathbb{R}^d$,
\begin{equation*}
   \sup_{ \tiny \begin{array}{c}
        f\in L_1+L_2 \\
        f \text{ radial}
   \end{array}
   }\frac{|f(y)|}{\norm{f(x) e^{\pi \alpha |x|^2}}_2 + \norm{\widehat{f}(\xi) e^{\pi \alpha |\xi|^2}}_2} \approx _{\alpha, d}  (1+|y|)^{\frac{2-d}{2}}e^{-\pi \alpha |y|^2}.
\end{equation*}
\end{theorem}

In particular, we note that,
for any radial function $f$,
 \eqref{vsp5} with $0<\alpha<1$ and $d\ge 2$ implies that $|f(x)|\lesssim e^{-\pi \alpha |x|^2}$.





\begin{proof}[Sketch of the proof of Theorem
\ref{radial}]
First, we recall that for any radial  function $f(x)=f_0(|x|)$ its Fourier transform
 is given by
\begin{equation}
\label{eq:radialfourier}
    \widehat{f}(|\xi|) =  \int_{0}^{\infty} f_0(x) L(x |\xi|)  x^{d-1} dx, \quad  L(z) :=c_d\frac{J_{\frac{d-2}{2}}(2 \pi z)}{z^{\frac{d-2}{2}}},
\end{equation}
 where $J_\alpha$ is the Bessel function of order $\alpha$ and $c_d$ is a normalization constant. This allows us to consider Theorem \ref{radial} as a version of the one-dimensional case of Theorem \ref{theorem:mainth} with $L(x |\xi|)$ playing the role of 
  $\exp ({2 \pi i x \xi})$.

  Second, for the "$\lesssim$" part we adapt the one-dimensional case of Lemmas \ref{prop:phi} or \ref{prop:multiphi} in the following way:

 \begin{lemma}

    \label{prop:phirad}
    Let $\delta>0, m \geq 0,$  and $ \max(1,\lceil m \rceil)\leq m_0 \in \mathbb{N}$.  Define
    \begin{equation}
    \begin{aligned}
        \label{eq:raddefphi}
        \phi_0(x)&:=  \mathbbm{1}_{[y-\delta,y+\delta]}(x)   K_{m_0}( x) \left|1+ \frac{x-y}{\delta}\right|^{2m},  \quad x \in \mathbb{R},\\
        \phi(x)&:=\phi_0(|x|), \quad x \in \mathbb{R}^d,
        \end{aligned}
     \end{equation}
     where $K_{m_0}$ satisfies the following modification of \eqref{eq:repro}:
     $$ Q(y)= \int_{[y-\delta,y+\delta]} K_{m_0}(x) Q(x)   \left|1+ \frac{x-y}{\delta}\right|^{2m} |x|^{d-1}dx,$$
     for any polynomial $Q$ of degree $m_0$.

    Then there exists a constant $C:=C(d)>0$
    such that, for any $y>\delta$ large enough, the following hold:
 \begin{enumerate}[label=\normalfont{(\arabic*)}]
    \item  $\displaystyle\left(\int_{y- \delta}^{y+ \delta}  |\phi_0(x)|^{2} \big|1+ \frac{x-y}{\delta}\big|^{- 2 m}x^{d-1}d x\right)^{\frac{1}{2}} \leq  C \left(\frac{y m_0}{ \delta}\right)^{\frac{1}{2}} y^{-\frac{d}{2}}; $


    \item \;
    $\displaystyle
    |\widehat{\phi}(|\xi|)| \leq C^{m_0};
    \quad \xi \in \mathbb{R}^d;
$

    \item
 \;
    $ \displaystyle \left | \widehat{\phi}(|\xi|) -  L(y|\xi|) \right| \leq C^{m_0} \big(\frac{\delta  |\xi|}{m_0
    }\big)^{m_0}  .$

    \end{enumerate}
    \end{lemma}
The proof of items (1) and (2) follows just as in Lemma \ref{prop:phi}.  To verify item (3), we use the fact that if $P_y(x)$ is the Taylor polynomial of degree $m_0-1$ at the point $y$ of the function $$x \mapsto L(x |\xi|), $$ we have
$$\left | P_y(x) -  L(x |\xi|) \right| \leq C^{m_0}
 \frac{(|x-y| \xi)^{m_0 }}{(m_0)!},$$ which follows from the representation $$c_d \frac{J_{\frac{d-2}{2}}(t)}{t^{\frac{d-2}{2}}}=\int_{-1}^1  (1-u^2)^{\frac{d-3}{2}} e^{-itu}du,$$ see \cite[\href{https://dlmf.nist.gov/10.9.E4}{(10.9.4)}]{NIST:DLMF} and the reproducing kernel property.

With Lemma \ref{prop:phirad} in hand and using the function $\phi$ defined in \eqref{eq:raddefphi}, it is straightforward to obtain the radial counterparts of the results of Section \ref{section:ifpart}.

Third, for the $\gtrsim$ part we modify Lemma \ref{lemma:bestapprox} as follows:

\begin{lemma}
\label{lemma:bestapprox2} Let $M,N>0$ and $N\in \mathbb{N}$. Assume that $y>0$ is large enough. Define

\begin{eqnarray*}
    \widetilde{E}_N(L(\xi y))&:=&\inf_{P \text{ of degree}\leq N}\Big( \int_{0}^{M} \left|L(\xi y)-P(\xi) \right|^2  \xi^{d-1} d \xi
\Big)^\frac12\\
&=&\Big( \int_{0}^{M} \left|L(\xi y)-P^*(\xi) \right|^2  \xi^{d-1} d \xi
\Big)^\frac12
\end{eqnarray*}
and
 \begin{equation}
    \begin{aligned}
        \label{eq:defh212}
    h_0(\xi)&:= \frac{1}{\widetilde{E}_N (L(\xi y))^{2}}\mathbbm{1}_{[0,M]}(\xi) \left(L(\xi y)-P^*(\xi) \right),  \quad \xi \in \mathbb{R},\\
        h(\xi)&:=h_0(|\xi|), \quad \xi \in \mathbb{R}^d.
        \end{aligned}
     \end{equation}
Then, the following hold:
\begin{enumerate}[label=\normalfont{(\arabic*)}]
    \item There exists $c>0$ small enough such that,  for $N< \lfloor cyM \rfloor$,

    $$\widetilde{E}_N(L(\xi y)) \approx_d M^{\frac{1}{2}} y^{-\frac{d-1}{2}};$$

    \item For $N<\lfloor cyM \rfloor$, the function $h$
    satisfies
    \begin{enumerate}
    \item[(i)]\; $\widehat{h}(y)=1;$
\item[(ii)] \;$
\norm{\widehat{h}}_2
\approx_d M^{-\frac{1}{2}} y^{\frac{d-1}{2}};$
\item[(iii)]\,
     $\displaystyle\int_{ \mathbb{R}^d}h(|\xi|) Q(\xi)  d \xi =0$ for any $d$-dimensional polynomial $Q$ of degree $\leq N$.
    \end{enumerate}

\end{enumerate}

\end{lemma}
To prove (1), we observe that $$ \Big( \int_{0}^{M} \left|L(\xi y) \right|^2  \xi^{d-1} d \xi
\Big)^\frac12\geq \widetilde{E}_N(L(\xi y))
\gtrsim_d
M^\frac{d-1}{2} \Big( \int_{\frac{M}{2}}^{M} \left|L(\xi y)-P^*(\xi) \right|^2   d \xi
\Big)^\frac12.$$ Hence, the result follows as in Lemma \ref{lemma:Ecd} with the help of the well-known  asymptotic formula for the Bessel function
   $$L(z) = k_d z^{- \frac{d-1}{2}}  \cos( 2\pi z - \theta_d) + O(z^{-\frac{d}{2}}), \quad k_d>0.$$
   For (2), we note that \textit{(i)} and \textit{(ii)} 
   follow from \eqref{eq:defh212} and   item (1), 
    respectively. To obtain \textit{(iii)}, we notice that if $Q(\xi)$ is a $d$-dimensional polynomial of degree $N$, then
   $$\int_{\mathbb{R}^d} h(|\xi|) Q(\xi) d \xi
   \approx_d \int_0 ^\infty h_0(\xi) \tilde {Q}(\xi) \xi^{d-1} d \xi=0,$$
 since
 $\tilde{Q}(\xi) =\int_{\mathbb{S}^{n-1}} Q( |\xi| \omega) d \omega$ is a one-dimensional polynomial of degree not greater than $N$.

 With Lemma \ref{lemma:bestapprox2} in hand, the lower bound in Theorem \ref{radial} can be proved in a similar way to that of Theorem \ref{theorem:mainth}; see Section \ref{4.3}.
\end{proof}

2.
 The approach described in item \textit{(iv)} of Remark \ref{remark:intro} can be used to deduce from inequality \eqref{eq:y2} the following generalization of the upper bound in Theorem \ref{theorem:mainth}:
 $$ |f(y)|^2 w(y)^2 \lesssim_d  (1+ |y|^{d})\left(\norm{f(x) w(x)}_2 ^2 + \norm{\widehat{f}(\xi) w(\xi)}_2 ^2\right),$$ where
\begin{equation}
\label{def:secw}
  w(x)^2 := \sum_{m>d/2}^\infty \frac{a_m}{m!} |x|^{2m}, \quad a_m\geq 0.
\end{equation}
Similarly, applying inequality \eqref{eq:m}, we obtain
\begin{equation}
\label{ineq:Cs}
    |f(y)|^2 \sum_{m>d/2}^\infty a_m C^{-2m} \lesssim_d  \norm{f(x) w(x)}_2 ^2 + \norm{\widehat{f}(\xi) w(\xi)}_2 ^2.
\end{equation}
Thus, if the sequence $(a_m)$ is such that \begin{equation}
\label{eq:trivialspace}
    \sum_{m>d/2}^\infty a_m C^{-2m}=\infty,
\end{equation} from inequality \eqref{ineq:Cs} we see that any function $f$ satisfying 
 $\norm{f(x) w(x) }_2  + \norm{\widehat{f}(\xi) w(\xi)}_2 <\infty$ vanishes identically.

In particular, by setting $a_m = C^{2m}$, we deduce that $\norm{f(x)e^{\frac{C^2}{2} |x|^2} }_2  + \norm{\widehat{f}(\xi) e^{\frac{C^2}{2} |\xi|^2}}_2 <\infty$ implies that $f\equiv 0,$ which is a weaker version of the Cowling--Price uncertainty principle
\eqref{cowling}.
It is interesting to note that for $f(x) =e^{-\pi |x|^2}$
\begin{equation*}
 \norm{f(x) w(x) }^2_2  + \norm{\widehat{f}(\xi) w(\xi)}^2_2\approx _d \sum^\infty_{m>d/2} \frac{a_m}{m!} \frac{\Gamma(m+\frac{d}{2})}{(2\pi) ^m} \approx _d \sum^\infty_{m> d /2 }
 \frac{a_m}{(2\pi)^{m}}
 m^{\frac{d}{2}-1}.
\end{equation*}
In particular, if \begin{equation}
\label{eq:gaussiannorm}
   \sum^\infty_{m> d /2 }
a_m (2\pi)^{-m} m^{\frac{d}{2}-1} <\infty,
\end{equation} then there exist non-zero functions satisfying  $\norm{f(x) w(x) }_2  + \norm{\widehat{f}(\xi) w(\xi)}_2 <\infty$.
Motivated by these observations, we pose
the following question:

{\it Open problem.}     For $w$ given by \eqref{def:secw}, set $S_w(D):=  \sum_{m> d/2 }^\infty a_m D^{-2m}$, $D>0$.
Define
$$E_w:=\left\{f \in L ^1(\mathbb{R}^d): \norm{f(x) w(x) }_2  + \norm{\widehat{f}(\xi) w(\xi)}_2<\infty\right\}
$$
    and $$G:= \Big\{D \in \mathbb{R}_+ : \text{for any $w$ given by \eqref{def:secw}}, \,\,\,\,S_w(D)=\infty \implies E_w = \{0\}\Big \}.$$
    Find $\inf G.$

Here, by \eqref{eq:trivialspace}, we know that $C\in G$, so $G$ is not empty; and, by \eqref{eq:gaussiannorm}, that $\inf G \geq \sqrt{2 \pi}$.

3. Using the results from Section 3, we now explicitly calculate   the supremum
 on the left-hand side of
\eqref{eq:derivatives--}. This result is of independent interest due to its relation to  point evaluation estimates in the Paley--Wiener space, which correspond to the case $\beta=0$, cf. \cite{quim}.


\begin{proposition}
\label{prop:betak}
Let $0<K<\infty$ and $0<\beta<1$. Then, $$\sup
\frac{|g(y)|}{N_{K,\beta}(g)} \approx_{K,\beta} (1+|y|)^{\frac{\beta}{2(1- \beta)}},$$ where the supremum
 on the left-hand side
 is taken over all
bandlimited  functions $g\in L_2(\mathbb{R})$ and
\begin{equation}
\label{eq:derivatives2}
N_{K,\beta}(g)^2:= \norm{g}_2^2 + \sum_{n=0}^\infty \frac{ K^{2n} |g^{(n)}(0)|^2}{ (n!)^{2 \beta}}.
\end{equation}
\end{proposition}
\begin{proof}[Sketch of the proof]
For the upper bound, let $y>0$.  Let $\phi$ be as in Lemma \ref{prop:phi}  with $m=0$,
$m_0:= \lceil Dy^{\frac{1}{1- \beta}} \rceil$,
 and $\delta:=y$ for a large enough  $D=D(\beta,K).$
 Then, using the reproducing formula \eqref{eq:repro}, we obtain
$$\left(1- \widehat{\phi}(-\xi)\right)e^{2 \pi i \xi y}=\int_{y-\delta}^{y+ \delta} \phi(x-y) (e^{2 \pi i \xi y} - e^{2 \pi i \xi x})dx = \sum_{n \geq m_0} c_n \xi ^n$$ with
$\displaystyle|c_n| \lesssim \frac{ (2 \pi \delta)^n}{n!}$. Then,
$$g(y)
= \int_{- \infty}^\infty \hat{g}(\xi) \left(1- \widehat{\phi}(-\xi)\right)e^{2 \pi i \xi y} d\xi +  \int_{- \infty}^\infty \hat{g}(\xi)  \widehat{\phi}(-\xi)e^{2 \pi i \xi y} d\xi.$$ Using that $\hat{g}$ is compactly supported and the Cauchy-Schwarz inequality, we arrive at
\begin{eqnarray*}
\left|\int_{- \infty}^\infty \hat{g}(\xi) \left(1- \widehat{\phi}(-\xi)\right)e^{2 \pi i \xi y} d\xi\right| &\leq& \sum_{n \geq m_0} \frac{|c_n|}{(2 \pi )^n} \left|g^{(n)}(0)\right|
\\&\lesssim&
N_{K, \beta}(g) \left(\sum_{n=m_0}^\infty \frac{ \delta^{2n}}{K^{2n} (n!)^{2(1-\beta)}}\right)^{\frac{1}{2}}.
\end{eqnarray*}
Thanks to the conditions satisfied by $\phi$ given in  Lemma \ref{prop:phi} (1), we obtain
$$  \int_{- \infty}^\infty \hat{g}(\xi)  \widehat{\phi}(-\xi)e^{2 \pi i \xi y} d\xi \leq \norm{\phi}_2 N_{K, \beta}(g)\lesssim (m_0/\delta)^{\frac{1}{2}}N_{K, \beta}(g),$$
completing the proof of the upper bound.

For the lower bound,
set
$M:=\lambda ^\beta y^{\frac{\beta}{1- \beta}}$, $N:= \lceil D \lambda y^{\frac{1}{1- \beta}} \rceil$, and
$\lambda:=(c/2D)^
\frac{1}
{1-\beta}$
for a large enough $D(\beta,K)$ and $c$ as in Lemma \ref{lemma:bestapprox}.
Let $P^*$ be the orthogonal projection of $\cos(2 \pi y \xi)$ onto the space of polynomials of degree $N$ on the interval $[-M,M]$ and set $$\hat{g}(\xi)=\mathbbm{1}_{[-M,M]}(\xi) \left(\cos(2 \pi y \xi) -P^*(\xi)\right).$$ Proceeding as in Lemma \ref{lemma:bestapprox}, we see that if $N\leq cMy$, $g(y) \approx \norm{g}_2 ^2 \approx M$, $g^{(n)}(0)=0$ for $0 \leq n \leq N$ and $|g^{(n)}(0)|\lesssim (2 \pi)^n M^{n+1}$ for any $n$. The proof is now complete.
\end{proof}

4. We conclude the paper by presenting the following quantitative  Cowling-Price uncertainty principle:
\begin{proposition}
\label{quantitativeCP}
    Let $1\leq p,q \leq \infty$.  Assume that $\norm{\widehat{f} e^{\pi \xi^2}}_q < \infty$. 
    Let 
    $$\tau:= \displaystyle
    \limsup_{|y|\to \infty} \frac{1}{2 \pi |y|}\log \left(\int_{- y}^{y} |f(x)|^p e^{p\pi  |x|^2} dx\right)^{\frac{1}{p}}.$$ 
    Then $\tau<\infty$ if and only if $\widehat{f}(\xi)=e^{-\pi \xi^2} \widehat{g}(\xi)$ with $\widehat{g} \in L^q$ of exponential type $\tau$. Moreover, if $\tau=0$, then ${f}\equiv 0$ if $q<\infty $; and ${f}(x)=Ce ^{-\pi x^2}$ if $q=\infty$.
\end{proposition}
\begin{proof}
    First, we assume that $\tau<\infty$. Let us  show that $\widehat{g}(\xi):=e^{\pi \xi^2} \widehat{f}(\xi)$ is of exponential type at most $\tau$.
    Observe that, for any $\lambda >1$, $$\displaystyle
    \limsup_{|y|\to \infty} \frac{1}{2 \pi |y|}\log \left(\int_{-\lambda y}^{\lambda y} |f(x)|^p e^{p\pi  |x|^2} dx\right)^{\frac{1}{p}}= \lambda \tau.$$ Hence, by item (\textit{v}) of Remark \ref{remark:intro},
     for any $\varepsilon>0$ we have $|f(x)|\lesssim_\varepsilon e^{- \pi x^2 + 2 \pi (\tau + \varepsilon) |x|}$. Thus $\widehat{f}$ is an entire function and, for $z\in \mathbb{C}$, \begin{equation}
     \label{eq:entire}
         |\widehat{f}(z)| \leq \norm{e^{2 \pi x \Im z }{f(x)}}_1 \lesssim _\varepsilon e^{\pi (\Im z)^2 + 2 \pi (\tau + \varepsilon) |\Im z|}.
     \end{equation}Furthermore, by item \textit{(i)} of Remark \ref{remark:intro}, 
     \begin{equation}
     \label{eq:rez}
        (1+ |\xi|)^{-\frac{1}{q}} e^{ \pi \xi^2} |\widehat{f} (\xi)| \lesssim  \norm{f(x) e^{\frac{\pi}{2} |x|^2}}_p + \norm{\widehat{f}(\xi) e^{\pi  |\xi|^2}}_q   .
     \end{equation}
     Set now $$G(z):=e^{2 \pi i (\tau + \varepsilon)z}\frac{\widehat{f}(z)e^{ \pi z^2} - \widehat{f}(0)}{z}.$$ 
     
     We are now going to use the Phragmén-Lindelöf principle to show that $G$ is bounded on the upper half-plane.
     
     By \eqref{eq:rez}, for real $\xi$, we have
    $|G(\xi)|\lesssim 1;$ and, by \eqref{eq:entire}, for $z$ with $\Im z \geq 0$ we have $|G(z)| \lesssim_{\varepsilon} e^{\pi (\Re z)^2 }.$ Thus, in light of the Hardy-Phragmén-Lindelöf theorem (see \cite[p. 134--135]{Havin}), $G$ is bounded on the upper half-plane, that is, $$\left |\frac{\widehat{f}(z)e^{ \pi z^2} - \widehat{f}(0)}{z}\right| \lesssim_\varepsilon e^{ 2 \pi (\tau + \varepsilon) |z|}$$ for $\Im z \geq 0$ and by symmetry the same holds for $\Im z \leq 0$. In conclusion, for any $\varepsilon>0$,
    $$|\widehat{f}(z)e^{ \pi z^2}| \lesssim _ \varepsilon e^{2 \pi (\tau + \varepsilon)|z|}\qquad \forall \,z\in \mathbb{C}.$$ 
    By the Paley-Wiener Theorem in $L_q$, 
    $ \widehat{f}(z)e^{ \pi z^2}$ is an $L_q$ function whose spectrum is contained in $[-\tau, \tau]$. In particular, if $\tau=0$, then $\widehat{f}\equiv 0$ provided that $q<\infty$; and $\widehat{f}(\xi)=Ce^{-\pi \xi^2}$ if $q=\infty$.

    For the converse, let $\widehat{g}$ 
     is an entire function of exponential type
 $\sigma$. Set $\widehat{f}(\xi) := \widehat{g}(\xi) e^{-\pi \xi^2}.$
    Observe that setting $h(x):=g(x)e^{-\pi x^2}$ we have $ f(x)=e^{-\pi x^2} \widehat{h}(ix)$ with $\widehat{h}$ of exponential type $\sigma$. Then, 
    $$ \log \left( \int_{- y}^{ y} |\widehat{h}(ix)|^p\right)^{\frac{1}{p}} \leq C_\varepsilon+2 \pi y  \sigma, 
    $$ whence we conclude that $\tau \leq \sigma$. 
\end{proof}

 \end{document}